\documentclass[12pt]{amsart}

\usepackage{hyperref}
\usepackage{amssymb}
\usepackage[bbgreekl]{mathbbol}
\usepackage{graphicx}
\usepackage{ifthen}
\usepackage{pict2e}
\usepackage{xargs}
\usepackage{xspace}
\usepackage{xcolor}
\usepackage{pgf,tikz}
\usepackage{pgfplots}
\usepackage{environ}
\usepackage{caption}
\usetikzlibrary{arrows}
\usepackage{enumitem}
\usepackage{xifthen}
\newcounter{dummy}
\makeatletter
\newcommand\myitem[1][]{\item[#1]\refstepcounter{dummy}\def\@currentlabel{#1}}
\makeatother

\makeatletter
\newsavebox{\measure@tikzpicture}
\NewEnviron{scaletikzpicturetowidth}[1]{%
	\def\tikz@width{#1}%
	\def\tikzscale{1}\begin{lrbox}{\measure@tikzpicture}%
		\BODY
	\end{lrbox}%
	\pgfmathparse{#1/\wd\measure@tikzpicture}%
	\edef\tikzscale{\pgfmathresult}%
	\BODY
}
\makeatother

\DeclareSymbolFontAlphabet{\mathbb}{AMSb}

\newcommand{\thistheoremname}{}
\newtheorem*{genericthm*}{\thistheoremname}
\newenvironment{namedthm*}[1]
{\renewcommand{\thistheoremname}{#1}%
	\begin{genericthm*}}
	{\end{genericthm*}}

\newcommand{\absolutevalue}[1]{|#1|}

\newcommand{\Bairespace}[1][]{
	\ifthenelse{\equal{#1}{}}{\functions{\N}{\N}}{\functions{#1}{\N}}
}
\newcommand{\bbL}{\mathbb{L}}
\newcommand{\bbX}{\mathbb{X}}
\newcommand{\Borelhomomorphism}{\le_B}
\newcommand{\Cantorspace}[1][]{
	\ifthenelse{\equal{#1}{}}{\functions{\N}{2}}{\functions{#1}{2}}
}
\newcommand{\completegraph}[1]{K_{#1}}
\newcommandx{\concatenation}[2][1 = undefined, 2 = undefined]{
	\ifthenelse{\equal{#1}{undefined}}{{}\smallfrown}{
		\ifthenelse{\equal{#2}{undefined}}{\bigoplus #1}{\bigoplus_{#1} #2}
	}
}

\newcommand{\continuoushomomorphism}{\le_c}
\newcommand{\definedterm}[1]{\emph{#1}}
\newcommand{\domain}[1]{\mathrm{dom}(#1)}
\newcommand{\from}{\colon}
\newcommandx{\functions}[3][3 =]{
	\ifthenelse{\equal{#3}{}}{#2^{#1}}{#2_{#3}^{#1}}
}

\newcommand{\Gzero}[1][]{
	\ifthenelse{\equal{#1}{}}
	{\mathbb{G}_0}
	{\mathbb{G}_{0,n}}
}
\newcommandx{\Hzero}[2][2 = undefined]{
	\ifthenelse{\equal{#2}{undefined}}
	{\mathbb{H}_{#1}}
	{\mathbb{H}_{#1, #2}}
}
\newcommandx{\intersection}[2][1 =, 2 =]{
	\ifthenelse{\equal{#1}{}}{\cap}{
		\ifthenelse{\equal{#2}{}}{\bigcap #1}{{\bigcap_{#1} #2}}
	}
}
\newcommand{\Lzero}[1][]{\ifthenelse{\equal{#1}{}}{\bbL_0}{L_{0, #1}}}
\newcommand{\Lzerospace}[1][]{\ifthenelse{\equal{#1}{}}{\bbX_0}{X_{0, #1}}}
\newcommand{\mathand}{\text{ and }}
\newcommand{\modulo}[1]{\ (\text{mod } 2)}
\newcommand{\N}{\mathbb{N}}

\newcommand{\pair}[2]{(#1, #2)}

\newcommandx{\product}[2][1 =, 2 =]{
	\ifthenelse{\equal{#1}{}}{\times}{
		\ifthenelse{\equal{#2}{}}{\prod #1}{{\prod_{#1} #2}}
	}
}

\newcommandx{\sequence}[2][2 = undefined]{
	\ifthenelse{\equal{#2}{undefined}}{(#1)}{
		(#1)_{#2}
	}
}

\newcommandx{\set}[2][2 = undefined]{
	\ifthenelse{\equal{#2}{undefined}}{\{ #1 \}}{
		\{ #1 \suchthat #2 \}
	}
}
\newcommandx{\sets}[3][3 =]{
	\ifthenelse{\equal{#3}{}}{[#2]^{#1}}{[#2]^{#1}_{#3}}
}

\newcommand{\suchthat}{\mid}

\renewcommand{\restriction}[2]{#1 \upharpoonright #2}

\newcommand{\triple}[3]{(#1, #2, #3)}
\newcommandx{\union}[2][1 =, 2 =]{
	\ifthenelse{\equal{#1}{}}{\cup}{
		\ifthenelse{\equal{#2}{}}{\bigcup #1}{{\bigcup_{#1} #2}}
	}
}

\newcommand{\Borel}{Bor\-el\xspace}

\newcommand{\Kechris}{Kech\-ris\xspace}
\newcommand{\Polish}{Po\-lish\xspace}
\newcommand{\Solecki}{So\-lec\-ki\xspace}
\newcommand{\Todorcevic}{To\-dor\-\v{c}ev\-ic\xspace}

\newtheorem{theorem}{Theorem}[section]
\newtheorem{lemma}[theorem]{Lemma}
\newtheorem{Claim}[theorem]{Claim}
\newtheorem{claim}[theorem]{Claim}

\newtheorem{proposition}[theorem]{Proposition}

\newtheorem{problem}[theorem]{Problem}

\newtheorem*{largegaps}{Theorem \ref{t:large_gaps}}

\theoremstyle{definition}

\newtheorem{definition}[theorem]{Definition}

\numberwithin{equation}{section}

\newcommand{\Z}{\mathbb{Z}}

\newcommand{\dlength}{\qopname\relax o{dilength}}
\newcommand{\length}{\qopname\relax o{length}}

\newcommand{\bd}{\begin{definition}}
	\newcommand{\ed}{\end{definition}}

\DeclareMathOperator{\dom}{dom}

\DeclareMathOperator{\mgs}{mgs}
\DeclareMathOperator{\dist}{dist}
\DeclareMathOperator{\didistance}{didist}

\newcommand{\distance}[3]{\ifthenelse{\isempty{#3}}{\dist(#1,#2)}{\dist^{#3}(#1,#2)}}
\newcommand{\didist}[3]{\ifthenelse{\isempty{#3}}{\didistance(#1,#2)}{\didistance^{#3}(#1,#2)}}
\newcommand{\digraph}[3]{\ifthenelse{\equal{#1}{b}}{\mathbb{#2}_{#3}}
	{{#2}_{#3}}}
\newcommand{\linegraph}[3]{\ifthenelse{\equal{#1}{b}}{\mathbb{#2}_{#3}}
	{#2_{#3}}}
\newcommand{\specialvertex}[2]{s_{#1}}
\newcommand{\underlyingspace}[3]{\ifthenelse{\equal{#1}{b}}{\mathbb{#2}_{#3}}
	{#2_{#3}}}
\newcommand{\distanceset}[2]{\ifthenelse{\isempty{#2}}{D(#1)}{D^{#2}(#1)}}
\newcommand{\Borelchromatic}[1]{\chi_B(#1)}

\newcommand{\connecteder}[1]{E_{#1}}
\newcommand{\concatt}{%
	\mathbin{\raisebox{1ex}{\scalebox{.7}{$\frown$}}}%
}

\newcommand{\directed}[1]{\mathbb{L}_{#1}}

\newcommand{\aStage}[1]{n^{#1}}
\newcommand{\aCode}[1]{b^{#1}}
\newcommand{\aCodeB}[1]{c^{#1}}
\newcommand{\aCodeD}[1]{d^{#1}}
\newcommand{\aMapV}[1]{\phi^{#1}}
\newcommand{\aMapE}[1]{\psi^{#1}}
\newcommand{\aBoundStage}[1]{\aCodeB{#1},\aStage{#1}}
\newcommand{\aDifferent}[2]{\aCodeB{#2},\aStage{#1},\aStage{#2}}
\newcommand{\aCodeStage}[1]{\aCode{#1},\aStage{#1}}
\newcommand{\imageConf}[2]{A(#1,#2)}
\newcommand{\BorelInvariant}[2]{B(#1,#2)}
\newcommand{\twocoloring}[2]{c_{#1,#2}}
\newcommand{\oddpair}[1]{odd $#1$-pair}
\newcommand{\lgth}{\length}

\pgfplotsset{soldot/.style={color=blue,only marks,mark=*}}
\usetikzlibrary{math}
\usetikzlibrary{arrows.meta}
\usetikzlibrary{shapes,decorations}
\usetikzlibrary{decorations.pathmorphing}

\begin{document}
	
	
	\begin{abstract}
	    We show that there is a \Borel graph on a standard \Borel space of \Borel chromatic number three that admits a \Borel homomorphism to every analytic graph on a standard \Borel space of \Borel chromatic number at least three. Moreover, we characterize the \Borel graphs on standard \Borel spaces of vertex-degree at most two with this property, and show that the analogous result for digraphs fails.
	\end{abstract}
	
	\author[R. Carroy]{Rapha\"{e}l Carroy}
	\address{
		Rapha\"{e}l Carroy \\
		Kurt G\"{o}del Research Center for Mathematical Logic \\
		Universit\"{a}t Wien \\
		W\"{a}hringer Stra{\ss}e 25 \\
		1090 Wien \\
		Austria
	}
	\email{raphael.carroy@univie.ac.at}
	\urladdr{
		http://www.logique.jussieu.fr/~carroy/indexeng.html
	}
	
	\author[B.D. Miller]{Benjamin D. Miller}
	\address{
		Benjamin D. Miller \\
		Kurt G\"{o}del Research Center for Mathematical Logic \\
		Universit\"{a}t Wien \\
		W\"{a}hringer Stra{\ss}e 25 \\
		1090 Wien \\
		Austria
	}
	\email{benjamin.miller@univie.ac.at}
	\urladdr{
		http://www.logic.univie.ac.at/benjamin.miller
	}
	
	\author[D. Schrittesser]{David Schrittesser}
	\address{
		David Schrittesser \\
		Kurt G\"{o}del Research Center for Mathematical Logic \\
		Universit\"{a}t Wien \\
		W\"{a}hringer Stra{\ss}e 25 \\
		1090 Wien \\
		Austria
	}
	\email{david.schrittesser@univie.ac.at}
	\urladdr{
		http://homepage.univie.ac.at/david.schrittesser/
	}
	
	\author[Z. Vidny\'{a}nszky]{Zolt\'{a}n Vidny\'{a}nszky}
	\address{
		Zolt\'{a}n Vidny\'{a}nszky \\
		Kurt G\"{o}del Research Center for Mathematical Logic \\
		Universit\"{a}t Wien \\
		W\"{a}hringer Stra{\ss}e 25 \\
		1090 Wien \\
		Austria
	}
	\email{zoltan.vidnyanszky@univie.ac.at}
	\urladdr{
		http://www.logic.univie.ac.at/~vidnyanszz77/
	}
	
	\thanks{The authors were supported in part by FWF Grants
		P28153 and P29999.}
	
	
	
	\title[Minimal definable graphs]{Minimal definable
		graphs of definable chromatic number at least three}
	
	\maketitle

	\section{Introduction}
	
	The investigation of definable chromatic numbers is a blooming field of research with numerous applications, as can be found in \cite{bernshteyn2016measurable,conleymeas,conley2014antibasis,csoka2016borel,lecomtezeleny2,lecomtezeleny,marksdet,marks,pequignot2017finite}. The survey \cite{kechris2015descriptive} contains many of the latest results.
	
	Recall that a \definedterm{digraph} on a set $X$ is an irreflexive set $G \subseteq  X^2$, and a \definedterm{graph} on $X$ is a symmetric digraph on $X$. A \definedterm{$\kappa$-coloring} of a digraph $G$ on $X$ is a map $c \from X \to \kappa$ such that $\pair{x}{y} \in G \implies c(x) \neq c(y)$ for all $x, y \in X$. We will be interested in digraphs on spaces $X$ which are endowed with a standard Borel structure. In this case, one may consider the \definedterm{\Borel chromatic number} of $G$, or $\Borelchromatic{G}$, defined as the least cardinal $\kappa$ that admits a standard \Borel structure with respect to which there is a \Borel $\kappa$-coloring of $G$. (Note that a standard Borel structure exists on $\kappa$ iff $\kappa \in \{0,1,2,\dots,\aleph_0, 2^{\aleph_0}\}$, and for each such $\kappa$ it is unique up to Borel isomorphism.)
	
	A \definedterm{homomorphism} from a digraph $G$ on $X$ to a digraph $G'$ on $X'$ is a map $\phi \from X \to X'$ such that $\pair{x}{y} \in G \implies \pair{\phi(x)}{\phi(y)} \in G'$ for all $x, y \in X$. When $G$ and $G'$ are digraphs on standard \Borel spaces, we write $G \Borelhomomorphism G'$ to indicate the existence of a \Borel homomorphism from $G$ to $G'$. Similarly, when $G$ and $G'$ are digraphs on \Polish spaces, we write $G \continuoushomomorphism G'$ to indicate the existence of a continuous homomorphism from $G$ to $G'$. It is easy to see that $G \leq_B G' \implies \Borelchromatic{G} \leq \Borelchromatic{G'}$. The \definedterm{complete graph} on $\kappa$ is given by $\completegraph{\kappa} = \set{\pair{\alpha}{\beta} \in \kappa^2}[\alpha \neq \beta]$. It is also easy to see that if $\kappa$ is endowed with a standard \Borel structure, then $\Borelchromatic{G} \leq \kappa \iff G \leq_B \completegraph{\kappa}$.
	
	The systematic investigation of \Borel chromatic numbers was initiated by \Kechris, \Solecki, and \Todorcevic \cite{kechris1999borel}. One of their primary successes was the isolation of a \Borel graph $\Gzero$ on $\Cantorspace$ of uncountable \Borel chromatic number that admits a continuous homomorphism to every analytic \Borel graph on a \Polish space of uncountable \Borel chromatic number. This result lies at the heart of a vast number of seemingly unrelated theorems in descriptive set theory (see, e.g., \cite{ClemensConleyMiller, ClemensLecomteMiller, miller2012graph, Miller:E3, benen2}), often yielding shorter, more elegant proofs and substantial generalizations. \Todorcevic and the fourth author \cite{toden} recently ruled out the most straightforward analogs of the $\Gzero$ dichotomy for graphs of \Borel chromatic number at least $n$, where $4 \le n \le \aleph_0$.
	
	We will introduce a \Borel graph $\Lzero$ that plays a role analogous to $\Gzero$ for graphs of \Borel chromatic number at least three:
	
	\begin{theorem} \label{t:main_intro}
	    Suppose that $G$ is an analytic graph on a Polish space. Then exactly one of the following holds:
	    \begin{enumerate}
	        \item The graph $G$ has \Borel chromatic number at most two.
	        \item There is a continuous homomorphism from $\Lzero$ to $G$.
	    \end{enumerate}
	\end{theorem}
	
	\noindent
	It is easy to see that there is no analogous finite basis in the case of finite graphs, where the notions of \Borel graph and \Borel chromatic number coincide with their classical counterparts.
	
	The graph $\Lzero$ can be described using an inverse limit-like construction as follows: Let $\Lzerospace[0]$ be a two-point set, let $\Lzero[0]$ be the unique connected graph on $\Lzerospace[0]$, and define $X_0=\Lzerospace[0]$. Given $n \in \N$, a finite set $\Lzerospace[n]$, and a tree $\Lzero[n]$ on $\Lzerospace[n]$ of vertex degree at most two, let $\Lzerospace[n+1]$ be the disjoint union of two copies of $\Lzerospace[n]$ with a set $X_{n+1}$ of cardinality $2n+2$, fix a point $s_n \in \Lzerospace[n]$ of $\Lzero[n]$-vertex degree one, fix a a tree $L_{n+1}$ on $X_{n+1}$ of vertex degree at most two, and let $\Lzero[n+1]$ be the graph on $\Lzerospace[n+1]$ whose restriction to each copy of $\Lzerospace[n]$ is the corresponding copy of $\Lzero[n]$, whose restriction to $X_{n+1}$ is $L_{n+1}$, and which connects the two copies of $s_n$ in $\Lzerospace[n+1]$ to distinct points of $X_n$ of $L_n$-vertex degree one (see the Figure below). Let $\pi_{n+1} \from \Lzerospace[n+1] \setminus X_{n+1}\to \Lzerospace[n]$ be the projection sending each point in one of the two copies of $\Lzerospace[n]$ within $\Lzerospace[n+1]$ to the corresponding point of $\Lzerospace[n]$. Let $\Lzerospace$ be the set of pairs of the form $\pair{n}{x}$, where $n \in \N$ and $x \in X_n \times \product[m > n][{\Lzerospace[m]}]$, such that $x(m) = \pi_{n+m+1}(x(m+1))$ for all $m \in \N$. Let $\Lzero$ be the graph on $\Lzerospace$ consisting of all pairs $\pair{\pair{n}{x}}{\pair{n'}{x'}} \in \Lzerospace \times \Lzerospace$ with the property that $\pair{x(m)}{x'(m)} \in \Lzero[m]$ for all $m \ge \max(n, n')$. We will give a slightly different description of this graph in \S\ref{s:preliminaries}.
	
	\begin{figure}
	\label{pic}
	\tikzstyle{every node}=[circle, draw, fill=black!50, inner sep=0.75pt]
	    \begin{scaletikzpicturetowidth}{.95\textwidth}
		    \begin{tikzpicture}[scale=\tikzscale]
        		\tikzmath{
			        function drawpath(\startx,\starty,\endx,\endy,\llength,\pathcolor) {
				        \llength = int(\llength);
				        int \i;
				        for \i in {0,..,\llength}{%
					        \j = \i/(\llength+1);%
					        \x = \startx + \j*(\endx-\startx);%
					        \y = \starty + \j*(\endy-\starty);%
					        \l = (\i+1)/(\llength+1);%
					        \z = \startx + \l*(\endx-\startx);%
					        \w = \starty + \l*(\endy-\starty);%
					        {
						        \draw (\x,\y) node {};
						        \draw (\z,\w) node {};
					        };
					        if (\pathcolor == 0) then
					        {
						        {
						        	\draw[green,bend right] (\z,\w) -- (\x,\y);
						        };
					        };
					        if (\pathcolor == 1) then
					        {
						        {
							        \draw[blue, bend left] (\z,\w) -- (\x,\y);
						        };
					        };
					        if (\pathcolor == 2) then
					        {
						        {
							        \draw[yellow] (\z,\w) -- (\x,\y);
						        };
					        };
					        if (\pathcolor == 3) then
					        {
						        {
							        \draw[red] (\z,\w) -- (\x,\y);
						        };
					        };
				        };
			        };
			        function drawingthegraph(\x,\y,\n,\s) {
				        \r = 1;
				        \t = 1;
				        if (\n == 0) then
				        {
					        \kx = 0;
					        \ky = 0;
				        };
				        if (\n == 1) then
				        {
					        \kx = 2;
					        \ky = 0;
				        };
				        if (\n == 2) then
				        {
					        \kx = 0;
					        \ky = 2;
				        };
				        if (\n == 3) then
				        {
					        \kx = 4;
					        \ky = 4;
				        };
				        if (\n == 4) then
				        {
					        \kx = 5;
					        \ky = 7;
				        };
				        if (\n == 0) then
				        {
					        \ky = 0;
				        };
				        if (\n > 0) then
				        {
					        \k = int(\n-1);
					        drawpath(\x,\y,(\x+\s*\kx),(\y+\s*\ky),2*\n-2,\n-1);%
					        drawingthegraph(\x,\y,\k,\s);
					        drawingthegraph((\x+\s*\kx),(\y+\s*\ky),\k,\s);
				        };
			        };
			        drawingthegraph(0,0,0,1);
			        drawingthegraph(3,0,1,1);
			        drawingthegraph(6,0,2,1);
			        drawingthegraph(9,0,2,1);
			        drawingthegraph(10,1,2,1);
			        drawpath(11,0,12,1,4,2);
			        \ratioo = 1;
			        \ratiou = 1;
			        \r = 3.5;
			        \t = 0;
			        \ri = 1;
			        \ti = 1;
			        drawingthegraph(9+\r,0+\t,2,\ratioo);
			        drawingthegraph(9+\r+\ri,0+\t+\ti,2,\ratioo);
			        drawpath(9+\r+\ratioo*2,0+\t+\ratioo*0,9+\r+\ri+\ratioo*2,0+\t+\ti+\ratioo*0,4,2);
			        \ru = 4.2;
			        \tu = -0.5;
			        \riu = 1;
			        \tiu = 1;
			        drawingthegraph(9+\ru,0+\tu,2,\ratiou);
			        drawingthegraph(9+\ru+\riu,0+\tu+\tiu,2,\ratiou);
			        drawpath(9+\ru+\ratiou*2,0+\tu+\ratiou*0,9+\ru+\riu+\ratiou*2,0+\tu+\tiu+\ratiou*0,4,2);
			        drawpath(9+\r+\ratioo*2,0+\t+\ratioo*2,9+\ru+\ratiou*2,\tu+\ratiou*2,6,3);
		        }
		    \end{tikzpicture}
		\end{scaletikzpicturetowidth}
	\caption{The first four stages of the construction of $\Lzero$.}
	
	\end{figure}
	
	Our proof of Theorem \ref{t:main_intro} splits into two parts: We first establish the existence of continuum-many $\Lzero$-like \Borel digraphs that serve as a basis for the analytic digraphs on \Polish spaces of \Borel chromatic number at least three under continuous homomorphism, and then show that the undirected versions of any of these digraphs admits a continuous homomorphism to the undirected version of any other.
	
	Suppose that $X$ is a set and $L$ is a graph on $X$ of vertex degree at most two. We say that a set $Y \subseteq  X$ has \definedterm{large gaps} if every $L$-component contains $L$-connected sets disjoint from $Y$ of arbitrarily large finite cardinality. When $X$ is a standard \Borel space, we say that $L$ has the \definedterm{large gap property} if there is a \Borel set $B \subseteq  X$ with large gaps that intersects every $L$-component. We say that $L$ has the \definedterm{large gap property modulo a two-colorable set} if there is an $L$-invariant \Borel set $M \subseteq  X$ such that $\restriction{L}{(X \setminus M)}$ has the large gap property and $\Borelchromatic{\restriction{L}{M}} \leq 2$. We also characterize the family of \Borel graphs $L$ on standard \Borel spaces of vertex degree at most two satisfying the analog of Theorem \ref{t:main_intro} in which the existence of a continuous homomorphism from $\Lzero$ to $G$ is replaced with the existence of a \Borel homomorphism from $L$ to $G$:
	
	\begin{theorem}
		\label{t:large_gaps}
		Suppose that $X$ is a standard \Borel space and $L$ is an acyclic \Borel graph on $X$ of vertex degree at most two. Then the following are equivalent:
		\begin{enumerate}
			\item There is a \Borel homomorphism from $L$ to every \Borel graph $G$ of \Borel chromatic number at least three.
			\item The graph $L$ has the large gap property modulo a two-colorable set.
		\end{enumerate}
		
	\end{theorem}
	
	An \definedterm{oriented graph} on a set $X$ is an antisymmetric digraph on $X$. Whereas the oriented analog of $\Gzero$ satisfies the analog of the \Kechris-\Solecki-\Todorcevic dichotomy for analytic digraphs, we also show that there is no such analog of Theorem \ref{t:main_intro}:
	
	\begin{theorem}
		\label{t:antibasis}
		Suppose that $G$ is an analytic digraph on a Polish space of \Borel chromatic number at least three. Then there is a sequence
		$\sequence{L_t}[t \in \Cantorspace]$ of \Borel oriented graphs on \Polish spaces of \Borel chromatic number three that admit continuous homomorphisms to $G$ but for which every analytic digraph on a standard \Borel space that admits a \Borel homomorphism to at least two distinct graphs of the form $L_t$ has \Borel chromatic number at most two.
	\end{theorem}

    One can view $\Lzero$ as being built via towers over a canonical acyclic graph $L$ on $\Cantorspace$ of vertex degree at most two that is not the graph of a \Borel function.
    In a future paper, we will establish a basis theorem for the analytic graphs on \Polish spaces of \Borel chromatic number at least three under the finer notion of injective continuous homomorphism. While the cardinality of the basis we will provide is necessarily (at least) that of the continuum, its elements are reminiscent of $\Lzero$, in that they too can be viewed as being built via towers, albeit this time over three canonical graphs: the graph $L$ over which $\Lzero$ is built, the graph of the odometer on $\Cantorspace$, and the graph of the unilateral shift on increasing $\N$-sequences of natural numbers (for a summary of the results, see, \cite{shortsum}).
	
	In \S\ref{s:preliminaries}, we collect the most important definitions and facts used in our arguments. In \S\ref{s:basis}, we give the first half of the proof of Theorem \ref{t:main_intro}. In \S\ref{s:large}, we give the second half and establish Theorem \ref{t:large_gaps}. In \S\ref{s:digraphs}, we establish our anti-basis result. In \S\ref{s:problems}, we discuss open problems.
	
	\section{Preliminaries and basic facts}
	\label{s:preliminaries}
	
	We refer the reader to \cite{kechrisbook} for general background on descriptive set theory. 
	
    For each ordered pair $\pair{x}{y}$, set $\pair{x}{y}^1 = \pair{x}{y}$ and $\pair{x}{y}^{-1} = \pair{y}{x}$. Define $B^{-1} = \set{\pair{x}{y}^{-1}}[\pair{x}{y} \in B]$ for all sets $B \subseteq  \functions{2}{X}$. Given a digraph $G$ on a set $X$ and $x,y \in X$, an (undirected) $G$-path from $x$ to $y$ is a pair $p = \pair{\sequence{x_0, \ldots, x_\ell}}{d_p}$ consisting of a finite sequence of vertices $\sequence{x_0,...,x_\ell}$ with $x_0 = x$ and $x_\ell = y$, and $d_p \in \set{\pm 1}^{\ell}$ such that $\pair{x_i}{x_{i+1}}^{d_p(i)} \in G$ for all $i < \ell$. In the case that $p$ is a $G$-path and $G$ is a graph, we will omit the second coordinate of $p$.

    For all $d \in \functions{<\N}{\set{\pm 1}}$, we use $\Sigma(d)$ to denote $\sum_{i \in \domain{d}} d(i)$. We set $\dlength(p) = \Sigma(d_p)$ and $\length(p) = \ell$ for the directed length and (undirected) length of $p$. Let $\dist_G(x, y)$ be the
    minimal length of a $G$-path from $x$ to $y$.
     
	It is easy to verify the next claim.
	
	\begin{claim}
		\label{c:distances} Let $G$ be an acyclic\footnote{Throughout the paper, the term ``acyclic'' will mean that there are no undirected cycles, that is, the graph $G \cup G^{-1}$ contains no cycles.} oriented graph on the space $X$, and $x,y \in X$. Then for any two $G$-paths $p$ and $p'$ from $x$ to $y$ we have $\dlength(p) = \dlength(p')$. 
	\end{claim}
	
	Thus, for an oriented acyclic graph $\digraph{}{G}{}$ on the space $X$, and $x,y \in X$ defining \definedterm{$\didist{x}{y}{\digraph{}{G}{}}$} to be the directed length of a path from $x$ to $y$ makes sense. If it is clear from the context, we will omit the superscript from $\distance{\cdot}{\cdot}{}$ and $\didist{\cdot}{\cdot}{}$.

	Note also that the parity of $\dlength(p)$ and $\length(p)$ are the same. So, when referring to the parity of the length of a path, we will always omit the word ``directed".

	Define an equivalence relation $\connecteder{\digraph{}{G}{}}$ on $X$ by letting $x\connecteder{\digraph{}{G}{}}y$ iff there exists a $\digraph{}{G}{}$-path from $x$ to $y$. The $\connecteder{\digraph{}{G}{}}$ equivalence classes will be called the \definedterm{connected components or components of $\digraph{}{G}{}$}. For standard definitions and facts from the theory of equivalence relations (e.g., smoothness, saturation, countability) see \cite{gao}. As usual, a set $S \subseteq X$ will be called $G$-invariant if it is $\connecteder{G}$-invariant.
	
	The \definedterm{restriction of $\digraph{}{G}{}$ ($\connecteder{\digraph{}{G}{}}$) to $B$}, in notation $\restriction{\digraph{}{G}{}}{B}$ ($\restriction{\connecteder{\digraph{}{G}{}}}{B}$), is the digraph $\digraph{}{G}{} \cap B^2$ (the equivalence relation $\connecteder{\digraph{}{G}{}} \cap B^2$) on $B$. A set $B \subseteq X$ is called \definedterm{$\digraph{}{G}{}$-independent}
	if $B^2 \cap \digraph{}{G}{}=\emptyset$.	
	
	\textit{Definition of $\linegraph{b}{L}{0}$-type graphs.} Now we outline a general scheme for constructing Borel graphs, the graph $\linegraph{b}{L}{0}$ will be a particular example of such a construction. First we define finitary approximations to our graphs, parametrized by a sequence $c \in \Bairespace$. For all $n \in \N$, let $L_n$ denote the graph on $\set{\sequence{0},
		\ldots, \sequence{n}}$ with respect to which $\sequence{i}$ and
	$\sequence{j}$ are neighbors if and only if $\absolutevalue{i - j} =
	1$. For the rest of the paper we fix a sequence $ \sequence{\specialvertex{n}{}}[n\in \N]$ given by  $\specialvertex{0}{}=(c(0))$ and $\specialvertex{n}{}=(0)^n \concatt (1)$, for $n>0$.
	Define graphs $\linegraph{}{L}{c,n}$ on $\underlyingspace{}{X}{c,n}=\union[m \le n][{\set{0,
			\ldots, c(m)} \times \Cantorspace[n-m]}]$ by setting $\linegraph{}{L}{c,0} = L_{c(0)}$ and  $\linegraph{}{L}{c,n+1}$ to be the acyclic connected graph
	containing $\{(v_i \concatenation \sequence{j})_{i < 2}|
	j < 2 \mathand \sequence{v_i}[i < 2] \in \linegraph{}{L}{c,n}\}$ and
	$L_{c(n+1)}$ in which $\pair{\specialvertex{n}{}}{0}$ is a neighbor of $\sequence
	{0}$, and $\sequence{c(n+1)}$ is a neighbor of $\pair{\specialvertex{n}{}}{1}$. 
	
	Now set
	$\underlyingspace{b}{X}{c}= \set{\triple{n}{k}{r} \in \N \times
		\N \times \Cantorspace}[k \le c(n)]$, define $\pi_{c,n} \from \underlyingspace{b}{X}{c}
	\intersection (\set{0, \ldots, n} \times \N \times \Cantorspace) \to
	\underlyingspace{}{X}{c,n}$ by
	$\pi_{c,n}(m, k, r ) = \sequence{k} \concatenation \restriction{r}{(n -
		m)}$ for all $n \in \N$, and let $\linegraph{b}{L}{c}$ be the graph on
	$\underlyingspace{b}{X}{c}$ consisting of all pairs of the form $\sequence{\triple
		{n_i}{k_i}{r_i}}[i < 2]$ such that $\sequence{\pi_{c,n}(n_i, k_i,
		r_i)}[i < 2] \in \linegraph{}{L}{c,n}$ holds $\forall n \geq \max(n_0, n_1)$.
	
	Recall that in the introduction we have described the graph $\linegraph{b}{L}{0}=\linegraph{b}{L}{c}$ with $c(0)=1$, and $c(n)=2n-1$ for $n>0$.
	
	\textit{Definition of $\linegraph{b}{L}{0}$-type oriented graphs.} We modify slightly the preceding construction, considering oriented finitary approximations, which yield oriented Borel graphs as limits. 
	
	An extra parameter is necessary to encode the orientation of the graphs. For all $n \in \N$ and $d\in \{-1,1\}^{k}$ with $k >n$, let $\digraph{}{L}{n}^d$ denote the oriented graph on $\set{\sequence{0},
		\ldots, \sequence{n}}$ containing $\pair{(i)}{(j)}^{d(\max\{i,j\})}$ whenever $\absolutevalue{i - j} =
	1$. 
	
	Let $n \in \{0,1,\dots, \aleph_0\}$. In order to ease the notation, we will call a pair $b=(c,d)$ an \definedterm{\oddpair{n}} if $c \from 1+n \to 2\N+1$, $d \from 1+n \to \{-1,1\}^{<\N}$, $\absolutevalue{d(k)}=c(k)+2$ for all $k \leq n$.

	Given an \oddpair{\aleph_0} $b=(c,d)$, define graphs $\digraph{}{L}{b,n}$ on $\underlyingspace{}{X}{c,n}$ by setting $\digraph{}{L}{b,0} = \digraph{}{L}{c(0)}^{d(0)}$ and letting $\digraph{}{L}{b,n+1}$ be the acyclic connected oriented graph
	containing $\{(v_i \concatenation \sequence{j})_{i < 2}|
	j < 2 \mathand \sequence{v_i}[i < 2] \in \digraph{}{L}{b,n}\}$ and
	$\digraph{}{L}{c(n+1)}^{d(n+1)}$, in which \[\pair{\pair{\specialvertex{n}{}}{0}}{\sequence
		{0})}^{d(n+1)(0)}\text{ and }\pair{\sequence{c(n+1)}}{\pair{\specialvertex{n}{}}{1}}^{d(n+1)(c(n+1)+1)}\] are edges. 
	Finally, let $\digraph{b}{L}{b}$ be the graph on
	$\underlyingspace{b}{X}{c}$ consisting of all pairs of the form $\sequence{\triple
		{n_i}{k_i}{r_i}}[i < 2]$ such that $\sequence{\pi_{c,n}(n_i, k_i,
		r_i)}[i < 2] \in \digraph{}{L}{b,n}$ holds $\forall n \geq \max(n_0, n_1)$.
	
	\textit{Basic observations.} Note that for any $c \in\N^\N$ and any \oddpair{\aleph_0} $b$, the definitions of $\linegraph{}{L}{c,n}$, $X_{c,n}$, and $\digraph{}{L}{b,n}$ depend only on $(c(i))_{i \leq n}$ and $(b(i))_{i \leq n}$, respectively. For $n'>n$ natural numbers define $\pi_{c,n,n'} \from X_{c,n'}
	\intersection \{(k,t) \suchthat t \in 2^{n'-m}, \text{ for some } m \leq n\} \to
	X_{c,n}$ by 
	$\pi_{c,n,n'}(k,t) = \sequence{k} \concatenation \restriction{t}{(n-m)}$, where $m$ is chosen such that $t \in 2^{n'-m}$. Observe that \[\restriction{\pi_{c,n,n'} \circ \pi_{c,n'}}{\dom(\pi_{c,n})}=\pi_{c,n}\] holds.  
	
	Let us use the abbreviation $\connecteder{c}$ for $\connecteder{\linegraph{b}{L}{c}}$. We list a number of useful basic observations about the family of digraphs and graphs defined above. 
	
	\begin{claim}
		\label{c:basic} Assume that $b=(c,d)$ is an \oddpair{\aleph_0}. Then 
		\begin{enumerate}
			\item \label{c:polish} $\underlyingspace{b}{X}{c}$ is a closed subset of $\N \times \N \times 2^\N$, hence it is a Polish space with the subspace topology.
			\item \label{c:symmetry} $\linegraph{b}{L}{c}=\digraph{b}{L}{b} \cup \digraph{b}{L}{b}^{-1}$, $\linegraph{}{L}{c,n}=\digraph{}{L}{b,n} \cup \digraph{}{L}{b,n}^{-1}$.
			\item \label{c:oddistance} If for some $n,k \in \N, \varepsilon <2,t \in 2^{<\N}$ we have $(k) \concatt t \concatt (\varepsilon) \in \linegraph{}{L}{c,n}$, then $(k) \concatt t \concatt (1-\varepsilon) \in \linegraph{}{L}{c,n}$ and $\distance{(k) \concatt t \concatt (\varepsilon)}{(k) \concatt t \concatt (1-\varepsilon)}{\linegraph{}{L}{c,n}}$ is odd.
			\item \label{c:connected} Let $(n,k,r),(n',k',r') \in \underlyingspace{b}{X}{c}$ with $n \leq n'$. Then $(n,k,r) \connecteder{c} (n',k',r')$ if and only if $r=t \concatt r^*$, $r'=t' \concatt r^*$ with $\absolutevalue{t}-\absolutevalue{t'}=n'-n$ for some $r^* \in 2^\N$, $t,t' \in 2^{<\N}$.
			\item \label{c:regularity} $\linegraph{b}{L}{c}$ is acyclic and is $2$-regular, except for a single vertex of degree $1$. 
			\item \label{c:meagersaturation}  If $B \subseteq \underlyingspace{b}{X}{c}$ is Borel and meager, then so is $[B]_{\connecteder{c}}$.
			\item \label{c:category} If $B \subseteq \underlyingspace{b}{X}{c}$ is Borel and non-meager, then $\Borelchromatic{\restriction{\linegraph{b}{L}{c}}{[B]_{\connecteder{c}}}}=3$.

			\item \label{c:largegap} If $\limsup_{n} c(n)=\infty$ then $\linegraph{b}{L}{c}$ has the large gap property.

		\end{enumerate}
		
	\end{claim}
	\begin{proof} It is immediate from the definition of the graphs $\digraph{b}{L}{b}$ that \eqref{c:polish} and \eqref{c:symmetry} holds, while \eqref{c:oddistance} follows from the fact that $c \in (2\N+1)^\N$.
		
		In order to see \eqref{c:connected} note that if $p=(x_0,\dots,x_l)$ is an injective $\linegraph{b}{L}{c}$-path, then for a large enough $m$ the sequence $(\pi_{c,m}(x_0),\dots,\pi_{c,m}(x_l))$ is an injective $\linegraph{}{L}{c,m}$-path. It follows from the fact that the graphs $\linegraph{}{L}{c,m}$ are acyclic that $\distance{(n,k,r)}{(n',k',r')}{\linegraph{b}{L}{c}} \geq \distance{\pi_{c,m}(n,k,r)}{\pi_{c,m}(n',k',r')}{\linegraph{}{L}{c,m}}$ holds for every large enough $m$. In particular, as the distance of the vertices in different copies of $\linegraph{}{L}{c,m-1}$ in $\linegraph{}{L}{c,m}$ is at least $c(m)+2$, we have that $(n,k,r) \connecteder{c} (n',k',r')$ if and only if $\pi_{c,m}(n,k,r)$ and $\pi_{c,m}(n',k',r')$ are in the same copy of $\linegraph{}{L}{c,m-1}$ in $\linegraph{}{L}{c,m}$ for every large enough $m$, which is equivalent to the right side condition in \eqref{c:connected}.
		
		For \eqref{c:regularity} observe that for every $n$ every degree in $\linegraph{}{L}{c,n}$ is at most $2$, hence the same must be true for $\linegraph{b}{L}{c}$. Also, it is easy to see that if the degree of a vertex $(n,k,r) \in \underlyingspace{b}{X}{c}$ is $<2$ then for every large enough $n'$ the degree of $\pi_{c,n'}(n,k,r)$ in $\linegraph{}{L}{c,n'}$ must be $<2$. It follows that this is only possible if $(n,k,r)=(0,0,(0)^\N)$. Finally, acyclicity follows from the acyclicity of $\linegraph{}{L}{c,n}$.
		
		From \eqref{c:connected} we get that $\connecteder{c}$ is the union of the graphs of the partial maps and their inverses of the following form:
		\[f_{n,n',k,k',t,t'}(n',k',t' \concatt r)=(n,k,t \concatt r),\]
		where $n' \geq n$, $\absolutevalue{t}-\absolutevalue{t'}=n'-n$.
		It is clear that the above partial maps are injective and preserve category. Thus, \[[B]_{\connecteder{c}} =\bigcup_{n,n',k,k',t,t'} f^{\pm 1}_{n,n',k,k',t,t'} (B)\] is also meager and Borel. 
		
		To see \eqref{c:category} first note that \eqref{c:regularity} implies $\Borelchromatic{\restriction{\linegraph{b}{L}{c}}{[B]_{\connecteder{c}}}}\leq 3$ using the standard fact that the maximal vertex degree $+1$ is an upper bound (see, e.g., \cite{kechris1999borel}).
		
		Assume that $B$ is a non-meager Borel set and that $c \from [B]_{\connecteder{c}} \to 2$ is a Borel $2$-coloring of $\restriction{\linegraph{b}{L}{c}}{[B]_{\connecteder{c}}}$. Then there exists an $i<2$ and a basic open set of the form $[(n,k,t)](=\{(n,k,r) \in \underlyingspace{b}{X}{c} \suchthat r \sqsupset t\})$ with $n,k \in \N, t \in 2^{<\N}$, such that $[(n,k,t)] \setminus (B \cap c^{-1}(i))$ is meager. Using \eqref{c:meagersaturation} we have that $[[(n,k,t)] \setminus (B \cap c^{-1}(i))]_{\connecteder{c}}$ is also meager, so we can pick a point $(n,k,r) \in [(n,k,t)] \cap B \cap c^{-1}(i)$ that does not belong to this meager set. Assume that $r=t\concatt (\varepsilon) \concatt r'$, then by \ref{c:connected} we have that $(n,k,t\concatt (1-\varepsilon) \concatt r') \in [(n,k,t)] \cap B \cap c^{-1}(i)$ holds, in particular $c(n,k,t\concatt (\varepsilon) \concatt r')=c(n,k,t\concatt (1-\varepsilon) \concatt r')=i$. As in the proof of \eqref{c:connected}, it follows that \[\distance{(n,k,t\concatt (\varepsilon) \concatt r') }{(n,k,t\concatt (1-\varepsilon) \concatt r')}{\linegraph{b}{L}{c}}=\]
		\[\distance{\pi_{c,n+\absolutevalue{t}+1}(n,k,t\concatt (\varepsilon) \concatt r') }{\pi_{c,n+\absolutevalue{t}+1}(n,k,t\concatt (1-\varepsilon) \concatt r')}{\linegraph{}{L}{c,n+\absolutevalue{t}+1}}=\]
		\[\distance{(k)\concatt t \concatt (\varepsilon)) }{(k) \concatt t \concatt (1-\varepsilon)}{\linegraph{}{L}{c,n+\absolutevalue{t}+1}},\] which is an odd number by \eqref{c:oddistance}. This contradicts the assumption that $c$ was a Borel $2$-coloring of $\restriction{\linegraph{b}{L}{B}}{[(n,k,r)]_{\connecteder{c}}}\subseteq \restriction{\linegraph{b}{L}{B}}{[B]_{\connecteder{c}}}$.

		Finally, for \eqref{c:largegap}, it is easy to verify that $B=\{(0,0,r) \in \underlyingspace{b}{X}{c} \suchthat r \in 2^\N\}$ witnesses the large gap property of $\linegraph{b}{L}{c}$, whenever $\limsup_n c(n)=\infty$.
		
		
		
	\end{proof}

	\begin{claim}
		\label{cl:basiclines}
		Assume that $\linegraph{}{L}{}, \linegraph{}{L}{}'$ are $\leq 2$-regular acyclic Borel graphs on standard Borel spaces $\underlyingspace{}{X}{}$, $\underlyingspace{}{X}{}'$.
		\begin{enumerate}
			
			\item \label{c:smooth} Let $A$ be an $L$-invariant analytic set so that $\restriction{\connecteder{\linegraph{}{L}{}}}{A}$ is smooth. Then $\Borelchromatic{\restriction{\linegraph{}{L}{}}{A}} \leq 2$.
			\item \label{c:ontohomo}
			Assume that $\phi$ is a Borel homomorphism from  $\linegraph{}{L}{}$ to $\linegraph{}{L}{}'$. Define $M=\{x \in \underlyingspace{}{X}{} \suchthat \phi \text{ mapping }[x]_{E_{\linegraph{}{L}{}}} \to [\phi(x)]_{E_{\linegraph{}{L}{}'}} \text{ is not onto}\}.$ Then $M$ is Borel and $\Borelchromatic{\restriction{\linegraph{}{L}{}}{M}} \leq 2$.  
		\end{enumerate}
	\end{claim}
	
	\begin{proof}
		
		In order to see \eqref{c:smooth} note that $\connecteder{\linegraph{}{L}{}}$ is countable, so smoothness is equivalent to the existence of an $\linegraph{}{L}{}$-invariant Borel partial mapping $x \to y_x$ so that $y_x \connecteder{\linegraph{}{L}{}}x $ holds, for every $x \in A$. Clearly, the map $c \from A \to 2$ defined by $c(x)=0$ iff $\distance{x}{y_x}{\linegraph{}{L}{}}$ is even, is a $2$-coloring of the graph $\restriction{\linegraph{}{L}{}}{A}$, such that for $i<2$ the sets $c^{-1}(i)$ are analytic. Using the analytic separation this yields that $\Borelchromatic{\restriction{\linegraph{}{L}{}}{A}} \leq 2$.

		For \eqref{c:ontohomo}, fix a Borel linear ordering $<$ on $X$. Since $\restriction{\linegraph{0}{L'}{}}{\phi([x]_{E_{\linegraph{}{L}{}}})}$ and $\restriction{\linegraph{0}{L'}{}}{[\phi(x)]_{E_{\linegraph{}{L'}{}}}}$ are $\leq 2$-regular acyclic connected graphs, there are one or two vertices in $[\phi(x)]_{E_{\linegraph{}{L'}{}}} \setminus \phi([x]_{E_{\linegraph{}{L}{}}})$ which have an $\linegraph{0}{L'}{}$-neighbor in $\phi([x]_{E_{\linegraph{}{L}{}}})$, let $y_x$ be the $<$-minimal such vertex. Now, similarly to $\eqref{c:smooth}$, letting $c(x)=0$ iff $\distance{\phi(x)}{y_x}{\linegraph{}{L'}{}}$ is even, shows that $\Borelchromatic{\restriction{\linegraph{}{L}{}}{M}} \leq 2$.	\end{proof}
	
	The following claim will be used to establish Theorem \ref{t:main_intro} for Borel graphs. 
	\begin{claim}
		\label{c:bhom}
		Assume that $G$ is a Borel graph on a standard Borel space $X$, $c \in (2\N+1)^{\N}$ and $(\phi_n)_{n \in\N}$ is a a sequence of Borel partial maps from $X$ to $\underlyingspace{}{X}{c,n}$ with the following properties for every $n\in \N$:
		\begin{enumerate}
			
			\item \label{p:union} $\dom(\phi_n) \subseteq \dom(\phi_{n+1})$ and $\bigcup_{n \in \N} \dom(\phi_n)=X$.
			\item \label{p:homomorphism} the map $\phi_n$ is a partial homomorphism from $G$ to $\linegraph{}{L}{c,n}$. 
			\item \label{p:compatible} $\phi_n=\restriction{\pi_{c,n,n+1}\circ \phi_{n+1}}{\dom(\phi_n)}.$ 
			
		\end{enumerate}	
		Then there exists a Borel homomorphism $\phi$ from $G$ to $\linegraph{b}{L}{c}$. 
		
		Moreover, 
		\begin{enumerate}
			\myitem[(4)] \label{p:continuity} if $X$ is Polish, for every $n \in \N$ the set $\dom(\phi_n)$ is open in $X$, and the maps $\phi_n$ are continuous, 
		\end{enumerate}
		then $\phi$ can be chosen to be continuous.
	\end{claim}
	
	\begin{proof}
		
		Let $x \in X$ be arbitrary and take $n^x_0$ to be minimal such that $x \in \dom(\phi_{n_0})$. For $n \geq n_0$ we have that $\phi_n(x)=(k_n) \concatt t_n$ for some $t_n \in 2^{n-m_n}$ and $k_n ,m_n\in \N$. By \eqref{p:compatible} for every $n\geq n_0$ the relations $k_n=k_{n+1}$, $m_n=m_{n+1}$, and $t_n \sqsubset t_{n+1}$ hold. Let $\phi(x)=(m_{n_0},k_{n_0},\bigcup_{n \geq n_0} t_n)$. Clearly, $\phi$ is a Borel map, we check that it is a homomorphism. Indeed, if $(x_i)_{i<2} \in G$ then by \eqref{p:homomorphism} letting $n \geq \max\{n^{x_i}_0 \suchthat i<2\}$ we have that $(\phi_n(x_i))_{i<2} \in \linegraph{}{L}{c,n}$. Notice that $\pi_{c,n}(\phi(x))=\phi_n(x)$, whenever $x \in \dom(\phi_n)$, so we obtain $(\pi_{c,n}(\phi(x_i)))_{i<2}=(\phi_n(x_i))_{i<2} \in \linegraph{}{L}{c,n}$, which verifies our claim by the definition of $\linegraph{b}{L}{c}$.
		
		Finally, one can easily check that the assumptions of \ref{p:continuity} of the claim yield the continuity of $\phi$.
	\end{proof}

	
	
	
	
	
	\section{A basis under continuous reducibility}
	\label{s:basis}
	In this section we construct a basis for Borel digraphs with Borel chromatic number $>2$. We will show these results in a somewhat greater generality than stated in the introduction, namely for analytic graphs defined on Hausdorff spaces. The proof relies on a slight modification of the idea behind the $\mathbb{G}_0$-dichotomy together with an observation about the Borel $2$-colorability of Borel digraphs, which is essentially summarized in Claims \ref{cl:terminal1}, \ref{cl:terminal_coloring}, and \ref{cl:coloring2} below.
	
	\begin{theorem}
		\label{t:g0style} Let $\digraph{}{G}{}$ be an analytic digraph on a Hausdorff space $X$. Then exactly one of the following holds:
		
		\begin{enumerate}
			\item $\Borelchromatic{\digraph{}{G}{}} \leq 2$.
			\item There exists an \oddpair{\aleph_0} $b$ so that $\directed{b}$ admits a continuous homomorphism to $\digraph{}{G}{}$. Moreover, for any $f \in \N^\N$ the pair $b=(c,d)$ can be chosen in such a way, so that for every $i \in \N$ we have $\Sigma(d(i))>f(i) \cdot \sum_{j<i} \absolutevalue{\Sigma(d(j))}$.
		\end{enumerate}
		
	\end{theorem}
	\begin{proof}
		The proof will follow the proof of the $\mathbb{G}_0$-dichotomy presented in \cite{millernote}.
		
		Fix a function $f \in \N^\N$. As $\digraph{}{G}{}$ is analytic, there exist a continuous surjection $\phi_{\digraph{}{G}{}} \from \N^\N \to \digraph{}{G}{}$ and a continuous map $\phi_X \from \N^\N \to X$ such that $\phi_X(\N^\N)$ is the union of the two projections of $\digraph{}{G}{}$ to $X$. By iteratively throwing away $\digraph{}{G}{}$-invariant sets restricted to which $\digraph{}{G}{}$ has a Borel $2$-coloring we define a decreasing sequence $(X^\alpha)_{\alpha<\omega_1}$ of analytic subsets of $X$. Let $X^0=\phi_X(\N^\N)$ and $X^\lambda=\bigcap_{\alpha<\lambda} X^{\alpha}$ if $\lambda$ is a limit ordinal.
		
		Let us now describe the successor stage.

		An \definedterm{approximation} is a quadruple  $a=(\aStage{a},\aCode{a},\aMapV{a},\aMapE{a})$, where $\aStage{a} \in \N, \aCode{a}=(\aCodeB{a},\aCodeD{a})$ is an \oddpair{\aStage{a}}, $\aMapV{a} \from \underlyingspace{}{X}{\aBoundStage{a}} \to \N^{\aStage{a}}$, and $\aMapE{a} \from \digraph{}{L}{\aCodeStage{a}} \to \N^{\aStage{a}}$. An approximation $a'$ said to \definedterm{one-step extend} $a$, if 
		\begin{enumerate}[label=(\alph*)]
			\item $\aStage{a'}=\aStage{a}+1$.
			\item $\aCodeB{a'} \sqsupset \aCodeB{a},\aCodeD{a'} \sqsupset \aCodeD{a}$.
			\item \label{con:directed_distance}
			$\Sigma(\aCodeD{a'}(\aStage{a'}))>f(\aStage{a}) \cdot \sum_{j \leq \aStage{a}} \absolutevalue{\Sigma(\aCodeD{a}(j))}$.
			\item \label{con:extendvertex} $\forall x \in \dom(\pi_{\aDifferent{a}{a'}}) \ \aMapV{a'}(x) \sqsupset \aMapV{a}\circ \pi_{\aCodeB{a'},\aStage{a},\aStage{a'}}(x)$.
			\item \label{con:extendedge} $\forall x,y \in \dom(\pi_{\aCodeB{a'},\aStage{a},\aStage{a'}})$ \[(x,y) \in \digraph{}{L}{\aCodeStage{a'}} \implies \aMapE{a'}(x,y) \sqsupset \aMapE{a}(\pi_{\aDifferent{a}{a'}}(x),\pi_{\aDifferent{a}{a'}}(y)).\]
		\end{enumerate}
		A \definedterm{configuration} is a quadruple of the form $\gamma=(\aStage{\gamma},\aCode{\gamma},\aMapV{\gamma},\aMapE{\gamma})$, where $\aStage{\gamma} \in \N$, $\aCode{\gamma}$ is an \oddpair{\aStage{\gamma}}, $\aMapV{\gamma} \from \underlyingspace{}{X}{\aBoundStage{\gamma}} \to \N^\N$, and $\aMapE{\gamma} \from \digraph{}{L}{\aCodeStage{\gamma}} \to \N^\N$ having the following property: for every $(x,y) \in \digraph{}{L}{\aCodeStage{\gamma}}$ 
		\begin{equation}
		\label{e:approxhom}
		(\phi_{\digraph{}{G}{}} \circ \aMapE{\gamma})(x,y)=(\phi_X \circ \aMapV{\gamma}(x),\phi_X \circ \aMapV{\gamma}(y)). \end{equation}
		
		A configuration $\gamma$ is said to be \definedterm{compatible} with an approximation $a$, if 
		
		\begin{enumerate}
			\item $\aStage{a}=\aStage{\gamma}$.
			\item $\aCode{a}=\aCode{\gamma}$.
			\item $\forall x \in \underlyingspace{}{X}{\aBoundStage{\gamma}} \ \aMapV{a}(x) \sqsubset\aMapV{\gamma}(x)$.
			\item $\forall (x,y) \in \digraph{}{L}{b^\gamma,n^\gamma} \ \aMapE{a}(x,y) \sqsubset\aMapE{\gamma}(x,y)$.
		\end{enumerate}
		We say that a configuration $\gamma$ is \definedterm{compatible} with a set $Y \subseteq X$, if $\phi_X \circ \aMapV{\gamma}(\underlyingspace{}{X}{\aBoundStage{\gamma}}) \subseteq [Y]_{\connecteder{\digraph{}{G}{}}}$. An approximation $a$ is \definedterm{$Y$-terminal} if no configuration is compatible with both $Y$ and a one step extension of $a$. Let \[\imageConf{a}{Y}=\{\phi_X \circ \aMapV{\gamma}(s_{\aStage{\gamma}}) \suchthat \gamma \text{ is compatible with } a \text{ and } Y\}.\] 
		\begin{lemma}
			\label{l:terminal}
			Suppose that $Y\subseteq X$ is an analytic set and $a$ is a $Y$-terminal approximation. Then there exists an $\digraph{}{G}{}$-invariant Borel set $\BorelInvariant{a}{Y} \supseteq [\imageConf{a}{Y}]_{\connecteder{\digraph{}{G}{}}}$ so that $\restriction{\digraph{}{G}{}}{\BorelInvariant{a}{Y}}$ has a Borel $2$-coloring, $\twocoloring{a}{Y}$.
		\end{lemma}
		We start with a series of claims. 
		
		\begin{claim}
			\label{cl:terminal1} Suppose that $A \subseteq X$ is an analytic set such that for every $x,y \in A$, every $\digraph{}{G}{}$-path from $x$ to $y$ has even (undirected or, equivalently, directed) length.
			Then there exists a $\digraph{}{G}{}$-invariant Borel set $B \supseteq [A]_{E_{\digraph{}{G}{}}}$ on which $\restriction{\digraph{}{G}{}}{B}$ admits a Borel $2$-coloring.
		\end{claim}
		\begin{proof}
			For $i <2$, define the $A_i \subseteq [A]_{E_{\digraph{}{G}{}}}$ as follows: let $x \in A_i$ if there exists a path of length $n$ from $x$ to some $y \in A$, where $n \equiv i \mod 2$. It is clear that the sets $(A_i)_{i<2}$ are analytic, their union covers $[A]_{E_G}$, and they are $\digraph{}{G}{}$-independent.  Note that their $\digraph{}{G}{}$-independence implies that $A_0 \cap A_1= \emptyset$. By the analytic separation theorem there exist Borel sets $B_i \supseteq A_i$ with $B_0 \cap B_1=\emptyset$. Define $c(x)=i \iff x \in B_i$ and let $C=\{x\in X \suchthat c \text{ is a $2$-coloring of } \restriction{\digraph{}{G}{}}{[x]_{\connecteder{\digraph{}{G}{}}}}\}$. Clearly, the sets $X \setminus C$ and $A_0 \cup A_1$ are disjoint, analytic, and  $\digraph{}{G}{}$-invariant. Hence, by \cite[Lemma 5.1]{harringtonkechris} there exists an $\digraph{}{G}{}$-invariant Borel set $B \supseteq A_0 \cup A_1$, with $B \cap (X \setminus C) = 
			\emptyset$. Then, $\restriction{c}{B}$ is a Borel $2$-coloring of $\restriction{\digraph{}{G}{}}{B}$. 
		\end{proof}
		\begin{Claim}
			\label{cl:terminal_coloring} 
			Let $A' \subseteq A \subseteq X$ be analytic sets and $d \in \Z$ be an odd number.
			Assume that for every $x' \in A'$ there exists an $x \in A$, such that there exists a $\digraph{}{G}{}$-path from $x'$ to $x$ of directed length $d$. Moreover, assume that every odd length $\digraph{}{G}{}$-path between elements of $A$ has directed length $\leq \absolutevalue{d}$. Then there exists an $\digraph{}{G}{}$-invariant Borel set $B \supseteq [A']_{\connecteder{\digraph{}{G}{}}}$ on which $\restriction{\digraph{}{G}{}}{B}$ admits a Borel $2$-coloring.
		\end{Claim}
		\begin{proof}
			Suppose that $d>0$, the other case is analogous. Let $A'_0=\{x' \in A' \suchthat$ there exists a $\digraph{}{G}{}$-path from $x'$ to some element of $A$ with odd negative directed length$\}$. We claim that $A'_0$ satisfies the assumptions of Claim \ref{cl:terminal1}. Assume that it is not the case, i.e., there exists $x',y' \in A'_0$ so that there exists a $\digraph{}{G}{}$-path of odd length between $x'$ and $y'$. As the directed length of odd length path is non-zero, we can assume (switching the roles of $x'$ and $y'$ if necessary) that there exists a path $p$ from $x'$ to $y'$ of positive odd directed length.  Then, using our assumptions on $A'_0$ and $A'$ there exist $z,w\in A$ and $\digraph{}{G}{}$-paths $q,r$, such that $q$ is a path from $z$ to $x'$, $r$ is a path from $y'$ to $w$ and $\dlength(q)>0$, $\dlength(r)=d$ and both of these numbers are odd. But then the path $q \concatt p \concatt r$ (i.e., the path $(q(0)\concatt p(0) \concatt r(0),d_q\concatt d_p \concatt d_r)$) connects $z$ with $w$ and $\dlength(q \concatt p \concatt r)>d+\dlength(p)$, and the former is an odd number $>d$, contradicting our assumption on $A$. 
			
			Now let $B_0$ be the invariant Borel superset of $[A'_0]_{\connecteder{\digraph{}{G}{}}}$ provided by Claim \ref{cl:terminal1} and define $A'_1=A' \setminus B_0$. Clearly, by the definition of $A'_0$ and as $A'_1 \subseteq A$, the set $A'_1$ also satisfies the requirements of Claim \ref{cl:terminal1}, so let $B_1 \supseteq [A'_1]_{\connecteder{\digraph{}{G}{}}}$ be the Borel set guaranteed. Then, it is easy to see from the invariance of $B_0$ and $B_1$ that $B=B_0 \cup B_1$ satisfies the requirements of the Claim. 
		\end{proof}
		
		\begin{Claim}
			\label{cl:coloring2} Let $A \subseteq X$ be an analytic set, and assume that there exists an $n \in \N$ such that whenever $x,y \in A$ and $p$ is a $\digraph{}{G}{}$-path of odd length from $x$ to $y$ then $\dlength(p) \leq n$. Then there exists an $\digraph{}{G}{}$-invariant Borel set $B \supseteq [A]_{\connecteder{\digraph{}{G}{}}}$ such that $\restriction{\digraph{}{G}{}}{B}$ admits a Borel $2$-coloring.  
		\end{Claim}
		\begin{proof}
			We prove this statement by induction on the minimal $n$ with this property. If $n=0$, then Claim \ref{cl:terminal1} yields the required conclusion.
			
			Now assume that we have shown the statement for every number $\leq n-1$. If $n>0$ is even, then it cannot be minimal, hence there is nothing to show. So we can assume that $n$ is odd. For $\varepsilon \in \{-1,1\}$ let  $A_{n,\varepsilon}=\{x \in A:$ there exists a $\digraph{}{G}{}$-path from $x$ to some $y \in A$ of directed length $\varepsilon \cdot n\}$. Now, we can apply Claim \ref{cl:terminal_coloring} to the sets $A_{n,\varepsilon}$, $A$ and $\varepsilon \cdot n$. This yields $\digraph{}{G}{}$-invariant Borel sets $B_{\varepsilon} \supseteq A_{n,\varepsilon}$ on which $\digraph{}{G}{}$ admits a Borel $2$-coloring. Note that if $x \in A \setminus (B_{-1}\cup B_1)$ then every odd length path between $x$ and an element of $A$ must have directed length $<n$. So, by the inductive hypothesis, we can find an invariant Borel set $B^{n-1} \supseteq [A \setminus (B_{-1}\cup B_{1})]_{\connecteder{\digraph{}{G}{}}}$, such that $\restriction{\digraph{}{G}{}}{B}$ admits a Borel $2$-coloring. Using the invariance of $B^{n-1},B_{-1}$, and $B_{1}$ again, we obtain that $\restriction{\digraph{}{G}{}}{B^{n-1} \cup B_{-1}\cup B_{1}}$ also admits a Borel $2$-coloring, which finishes the proof.
		\end{proof}
		
		\begin{proof}[Proof of Lemma \ref{l:terminal}]
			
			By definition, the set $\imageConf{a}{Y}$ is analytic.
			If there exists an $n \in \N$ such that every path $p$ of odd length between vertices from $\imageConf{a}{Y}$ have directed length  $\leq n$ then Claim \ref{cl:coloring2} yields the $\digraph{}{G}{}$-invariant Borel set $\BorelInvariant{a}{Y} \supseteq [\imageConf{a}{Y}]_{\connecteder{\digraph{}{G}{}}}$, and a Borel $2$-coloring $\twocoloring{a}{Y}$ of $\restriction{\digraph{}{G}{}}{\BorelInvariant{a}{Y}}$.
			
			So, assume that such an $n$ does not exist, we will show that $a$ is not $Y$-terminal.  Using this assumption for $n=f(\aStage{a}) \cdot \sum_{j\leq\aStage{a}} \absolutevalue{\Sigma(\aCodeD{a}(j))}$ we obtain two  configurations $(\gamma_j)_{j<2}$ compatible with $a$ and $Y$, a $\digraph{}{G}{}$-path of odd length $p=((x_0,\dots,x_{m+2}),d_p)$ with $\dlength(p)>f(\aStage{a}) \cdot \sum_{j\leq\aStage{a}} \absolutevalue{\Sigma(\aCodeD{a}(j))}$ such that $x_0=(\phi_X \circ \aMapV{\gamma_0})(s_{\aStage{a}})$ and $x_l=(\phi_X \circ \aMapV{\gamma_1})(s_{\aStage{a}}).$ Pick $r_0,\dots,r_{m+2}\in \N^\N$ and  $e_0,\dots,e_{m+1} \in \N^\N$ so that
			\begin{itemize}
				\item $r_0=\aMapV{\gamma_0}(s_{\aStage{a}})$, $r_{m+2}=\aMapV{\gamma_1}(s_{\aStage{a}})$,
				
				\item $\forall j \leq m+2 \ \phi_X(r_j)=x_j$,
				\item  $\forall j<m+2 \ \phi_{\digraph{}{G}{}}(e_j)=(x_{j},x_{j+1})^{d_p(j)}$.
			\end{itemize}   
			
			We define a configuration $\delta$ as follows: let $\aStage{\delta}=\aStage{a}+1,\aCode{\delta}=(\aCodeB{\delta},\aCodeD{\delta})=( \aCodeB{a} \concatt m, \aCodeD{a} \concatt  d_p)$, and define $\aMapV{\delta}:\underlyingspace{}{X}{\aBoundStage{\delta}} \to \N^\N$ by
			\[
			\begin{cases}
			\aMapV{\delta}(x \concatt (j))=\aMapV{\gamma_j}(x),&\text{ for } x \in \underlyingspace{}{X}{\aBoundStage{a}},j<2.\\
			\aMapV{\delta}((j))=r_{j+1}, &\text{ for } j  \leq m.
			\end{cases}
			\]
			Finally, define $\aMapE{\delta}:\digraph{}{L}{\aCodeStage{\delta}} \to \N^\N$ by 
			\[
			\begin{cases}
			\aMapE{\delta}(x \concatt (j), y \concatt (j))=\aMapE{\gamma^j}(x,y),&\text{ for } (x,y) \in \digraph{}{L}{\aCodeStage{a}},j<2.\\
			\aMapE{\delta}((s_{\aStage{a}} \concatt (0),(0))^{d(0)})=e_{0}.\\
			\aMapE{\delta}(((m),s_{\aStage{a}} \concatt (1))^{d(m+1)})=e_{m+1}.\\

			\aMapE\delta((j,j+1)^{d(j+1)})=e_{j+1}, &\text{ for } j \leq m-1.
			\end{cases}
			\]
			
			It is not hard to check that $\delta$ is a configuration. Moreover, as $\gamma_0$ and $\gamma_1$ are compatible with $Y$, so is $\delta$. Finally, using the fact that $\Sigma(d)=\dlength(p)>f(\aStage{a}) \cdot \sum_{j\leq\aStage{a}} \absolutevalue{\Sigma(\aCodeD{a}(j))}$, one can deduce that there exists a unique one-step extension $a'$ of $a$, that is compatible with $\delta$. This contradicts the assumption that $a$ was $Y$ terminal.
\end{proof}

		Define \[X^{\alpha+1}=X^{\alpha} \setminus \bigcup_{a \text{ is $X^\alpha$ terminal, }} \BorelInvariant{a}{X^{\alpha}}.\] Since there are only countably many possible approximations, and $X^0$ is an analytic set, the sets $X^\alpha$ are analytic for every $\alpha<\omega_1$. Note also that each $X^\alpha$ is  $\digraph{}{G}{}$-invariant.
		
		\begin{lemma}
			\label{l:iterate}
			Assume that $\alpha<\omega_1$ and $a$ is an approximation that is not $X^{\alpha+1}$-terminal. Then $a$ has a one-step extension that is not $X^\alpha$-terminal. 
		\end{lemma}
		\begin{proof}
			Let $a'$ be a one-step extension of $a$ for which there exists a configuration $\gamma$ compatible with $X^{\alpha+1}$ and $a'$. Then $\emptyset \not =(\phi_X \circ \phi^\gamma)(\underlyingspace{}{X}{\aBoundStage{\gamma}}) \subseteq [X^{\alpha+1}]_{\connecteder{\digraph{}{G}{}}}=X^{\alpha+1}$, but if $a'$ was $X^{\alpha}$-terminal, then $[(\phi_X \circ \aMapV{\gamma})(\underlyingspace{}{X}{\aBoundStage{\gamma}})]_{\connecteder{\digraph{}{G}{}}} \subseteq [\imageConf{a'}{X^{\alpha}}]_{\connecteder{\digraph{}{G}{}}}$ would be covered by $\BorelInvariant{a'}{X^\alpha}$, contradicting the definition and the $\digraph{}{G}{}$-invariance of  $X^{\alpha+1}$.
		\end{proof}
		Note that the set of $X^\alpha$-terminal approximations increases as $\alpha$ increases, and there are only countably many approximations. Thus, we can fix an $\alpha<\omega_1$ so that the $X^{\alpha}$-terminal and $X^{\alpha+1}$-terminal approximations are the same. 
		
		\begin{lemma}
			\label{l:terminalalpha}
			If every approximation is $X^{\alpha+1}$-terminal, then $\digraph{}{G}{}$ has a Borel $2$-coloring.
		\end{lemma}
		\begin{proof}
			Observe first that $X^{\alpha+1}$ is $\digraph{}{G}{}$-independent: otherwise, if $(x,y) \in \digraph{}{G}{} \cap (X^{\alpha+1})^2$, then there exists a configuration $\gamma$ with $\aCodeB{\gamma}=(1)$ compatible with $\{x,y\}$. Consequently, there exists an approximation $a$ that is compatible with $\gamma$ and $X^{\alpha+1}$. Then, $a$ is $X^{\alpha+1}$-terminal, so $x,y \in [\imageConf{a}{X^{\alpha+1}}]_{\connecteder{\digraph{}{G}{}}}$, but then $a$ is an $X^{\alpha}$-terminal approximation as well, so $x,y \in [\imageConf{a}{X^{\alpha}}]_{\connecteder{\digraph{}{G}{}}} \subseteq \BorelInvariant{a}{X^\alpha}$, contradicting the definition of $X^{\alpha+1}$. 
			
			Moreover, $X^{\alpha+1} \subseteq X^0$ is $\digraph{}{G}{}$-independent and $\digraph{}{G}{}$-invariant, so by the definition of $X^0$ it must be empty.
			
			Let $e: \{(a,\beta):a \text{ is $X^\beta$ terminal}, \beta \leq \alpha\} \to \N$ be an injection and let $\twocoloring{a}{X^\beta} $ be the Borel $2$-coloring of $\restriction{\digraph{}{G}{}}{\BorelInvariant{a}{X^\beta}}$ given by Lemma \ref{l:terminal}, for $(a,\beta) \in \dom(e)$. If $x \in X$, define $c(x)=$ 
			\[\begin{cases}
			\twocoloring{a}{X^\beta}(x), &\text{ if } e(a,\beta) \text{ is minimal such that } x \in \BorelInvariant{a}{X^{\beta}}\\
			0, &\text{ if } x \not \in  \bigcup_{(a,\beta) \in \dom(e)} \BorelInvariant{a}{X^\beta}.
			
			\end{cases}
			\]
			It is easy to check that $c$ is a Borel map and it is defined on $X$, while the $\digraph{}{G}{}$-invariance of the sets $\BorelInvariant{a}{X^\beta}$ implies that $c$ is a $2$-coloring.
		\end{proof}
		
		Now we are ready to finish the proof of Theorem \ref{t:g0style}. Assume $\Borelchromatic{\digraph{}{G}{}}>2$. Then, by Lemma \ref{l:terminalalpha} there exists an approximation that is not $X^{\alpha+1}$-terminal. Clearly, we can find such an $a_0$ with $\aStage{a_0}=0$. By applying Lemma \ref{l:iterate} recursively, we obtain one-step extensions $a_{n+1}$ of $a_n$ which are not $X^\alpha$-terminal, with $\aStage{a_n}=n$. Define $\aCode{}=(\aCodeB{},\aCodeD{})=\bigcup_n \aCode{a_n}$, $\phi:\underlyingspace{b}{X}{\aCodeB{}} \to \N^\N$, and $\psi: \digraph{b}{L}{\aCode{}} \to \N^\N$ by letting $\phi(m,k,r)=\bigcup_{n \geq m}\aMapV{a_n}(\pi_{\aBoundStage{}}(m,k,r))$ and for $(m_i,k_i,r_i)_{i<2} \in \digraph{b}{L}{\aCode{}}$ let 
		$\psi((m_i,k_i,r_i)_{i<2})=\bigcup_{n \geq m_0,m_1}\aMapE{a_n}((\pi_{\aBoundStage{}}(m_i,k_i,r_i))_{i<2})$.
		It follows from the fact that $a_{n+1}$ one-step extends $a_n$ (using conditions \ref{con:extendvertex}, \ref{con:extendedge}, and the fact that $\restriction{\pi_{c,n,n'} \circ \pi_{c,n'}}{\dom(\pi_{c,n})}=\pi_{c,n}$) that $\phi$ and $\psi$ are well-defined.

		
		Now, we check that $\phi_X \circ \phi$ is a continuous homomorphism of $\digraph{b}{L}{b}$ to $\digraph{}{G}{}$. The continuity of this mapping is clear from its definition, we check that it is a homomorphism. To this end, let $(x_0,x_1) \in \digraph{b}{L}{b}$ with $x_i=(n_i,k_i,r_i)$, for $i<2$. We claim that  
		\begin{equation}
		\label{e:homomorphism}
		(\phi_{\digraph{}{G}{}} \circ \aMapE{})(x_0,x_1)=((\phi_X \circ \aMapV{})(x_0),(\phi_X \circ \aMapV{})(x_1)), 
		\end{equation}
		which is clearly sufficient, as the left side is an element of $\digraph{}{G}{}$. We show that if $U$ and $V$ are open neighborhoods of $(\phi_{\digraph{}{G}{}} \circ \aMapE{})(x_0,x_1)$ and $((\phi_X \circ \aMapV{})(x_0),(\phi_X \circ \aMapV{})(x_1))$, then $U \cap V \not = \emptyset$.
		
		By the definition of $\digraph{b}{L}{b}$ we have that $(\pi_{c,n}(x_i))_{i<2} \in \digraph{}{L}{b,n}$ for every $n \geq \max(n_0,n_1)$. Thus, using the continuity of $\aMapV{}, \aMapE{}, \phi_{\digraph{}{G}{}}$, and $\phi_X$ we can find an $n \geq \max(n_0,n_1) $ so large that $U \supseteq \phi_{\digraph{}{G}{}}( [\aMapE{a_n}((\pi_{c,n}(x_i))_{i<2})])$ and $V \supseteq \phi_X([\aMapE{a_n}\circ\pi_{c,n}(x_0)]) \times \phi_X([\aMapV{a_n}\circ \pi_{c,n}(x_1)])$. 
		
		Let $\gamma$ be a configuration compatible with $a_n$.
		Then by \eqref{e:approxhom} we have that 
		\begin{multline}
		\label{e:configuration}
		(\phi_{\digraph{}{G}{}} \circ \aMapE{\gamma})((\pi_{c,n}(x_i))_{i<2})=\\
		((\phi_X \circ \aMapV{\gamma})(\pi_{c,n}(x_0)),(\phi_X \circ \aMapV{\gamma})(\pi_{c,n}(x_1)).
		\end{multline}

		Then from the compatibility of $\gamma$ and $a_n$ it follows that
		\begin{multline*}
		((\phi_X \circ \aMapV{\gamma})(\pi_{c,n}(x_0)),(\phi_X \circ \aMapV{\gamma})(\pi_{c,n}(x_1))) \in\\
		\phi_X([\aMapV{a_n}\circ\pi_{c,n}(x_0)]) \times \phi_X([\aMapV{a_n}\circ \pi_{c,n}(x_1)]) \subseteq V
		\end{multline*}
		and
		\[
		\phi_{\digraph{}{G}{}}\circ \aMapE{\gamma} ((\pi_{c,n}(x_i))_{i<2}) \in \phi_{\digraph{}{G}{}}( [\aMapE{a_n}((\pi_{c,n}(x_i))_{i<2})]) \subseteq U,
		\]
		which together with \eqref{e:configuration} implies $U \cap V \not = \emptyset$, finishing the proof of Theorem \ref{t:g0style}.
	\end{proof}

	\section{Large gaps}
	\label{s:large}
	
	In this section, we complete the proof of Theorem \ref{t:main_intro}, and prove Theorem \ref{t:large_gaps}. Note that graphs (rather than digraphs) will be considered. Let $\linegraph{}{L}{}$ be a graph on the space $X$, and assume that $B \subset X$. The minimal cardinality of an $\restriction{\linegraph{}{L}{}}{X \setminus B}$-component will be denoted by $\mgs(B)$. 
	
	We start with an easy observation.
	\begin{claim}
		\label{cl:large}
		Let $L$ be a $\leq 2$-regular acyclic Borel graph on a standard Borel space $X$ that has the large gap property. Then there exists an increasing sequence $(B_n)_{n \in \N}$ of Borel subsets of $X$, such that $\bigcup_{n \in \N} B_n$ is $\linegraph{}{L}{}$-invariant, $\restriction{\connecteder{\linegraph{}{L}{}}}{X \setminus \bigcup_{n \in \N} B_n}$ is smooth, for every $n \in \N$ the $\restriction{L}{B_n}$-components are finite, and $\mgs(B_n) \to \infty$, as $n \to \infty$.  
	\end{claim}
	\begin{proof}
		Let $B$ be a Borel set witnessing the large gap property of $L$.
		The graph $\linegraph{}{L}{}$ restricted to an $\linegraph{}{L}{}$-component is an infinite, connected, $\leq 2$-regular graph, which can be partitioned to disjoint $\restriction{\linegraph{}{L}{}}{X \setminus B}$-components. Let $S_0$ be the union of those $\linegraph{}{L}{}$-components which
		\begin{itemize}
			\item contain an infinite $\restriction{\linegraph{}{L}{}}{X \setminus B}$-component or
			\item the $\limsup$ of the cardinality of the $\restriction{\linegraph{}{L}{}}{X \setminus B}$-components is finite in one of the directions.
		\end{itemize}
		It follows from the choice of $B$ that $S_0$ is Borel and $\restriction{\connecteder{\linegraph{}{L}{}}}{S_0}$ is smooth.
		Let \begin{multline*}B_{n}=B \cup \{x\in X \setminus S_0: \text{the $\restriction{\linegraph{}{L}{}}{X \setminus B}$-component of $x$ has size $< n$}\}.\end{multline*}
		Clearly, the sets $B_n$ are increasing, $X \setminus S_0=\bigcup_{n \in \N} B_n$, $\mgs(B_n) \geq n$. Finally, note that if the $\restriction{\linegraph{}{L}{}}{B_n}$-component of $x$ was infinite, then the cardinality of $\restriction{\linegraph{}{L}{}}{X \setminus B}$-components would be bounded by $n$ in some of the directions in the $L$-connected component of $x$, in other words $x \in S_0$ would hold.
	\end{proof}
	
	The next proposition is the essence of the argument. 
	
	\begin{proposition}
		\label{p:large_homo} 
		
		Assume that $c \in (2\N+1)^{\N}$. 
		\begin{enumerate}
			\item \label{p:borel_homo} Let $L$ be a $\leq 2$-regular acyclic Borel graph on the standard Borel space $X$. Assume that $(B_n)_{n \in \N}$ is an increasing sequence of Borel subsets of $X$ with $\bigcup_{n \in \N} B_n=X$, $\mgs(B_n) \to \infty$, and for every $n$ the $\restriction{\linegraph{}{L}{}}{B_n}$-components are finite. Then $L \leq_B \linegraph{b}{L}{c}$.
			\item \label{p:continuous_homo} If $c^0 \in (2\N+1)^{\N}$, $c^0(n) \to \infty$, then $\linegraph{b}{L}{c^0} \leq_c \linegraph{b}{L}{c}$.
		\end{enumerate}
		
	\end{proposition}
	
		Our strategy is to inductively define sequences $k_n \in \N$ and $\phi_n:B_{k_n} \to \underlyingspace{}{X}{c,n}$, and appeal to Claim \ref{c:bhom}.
		
		We start with the key lemma.

		\begin{lemma}
			\label{l:inductive}
			\begin{enumerate}
				\item \label{p_lemma:borel}
				Let $B \subseteq B' \subseteq X$ be Borel, $n \in \N$ be given with the properties that $\mgs(B) > 2\cdot \lgth(\linegraph{}{L}{c,n+1})$, every component of $\restriction{\linegraph{}{L}{}}{B}$ is finite, and $\phi$ is a Borel homomorphism from $\restriction{\linegraph{}{L}{}}{B}$ to $\linegraph{}{L}{c,n}$. Then there exists a homomorphism $\phi'$ from $\restriction{\linegraph{}{L}{}}{B'}$ to $L_{c,n+1}$ so that  $\restriction{\pi_{c,n,n+1} \circ \phi'}{B} =\phi$ holds.
				
				\item \label{p_lemma:continuous} If moreover, $\linegraph{}{L}{}=\linegraph{b}{L}{c^0}$,  $B=\{(l,m,r)\in \underlyingspace{b}{X}{c^0}: l <k\}$, $B'=\{(l,m,r)\in \underlyingspace{b}{X}{c^0}: m <k'\}$ for $k<k'$, and $\phi$ is continuous then $\phi'$ can be taken to be continuous.
			\end{enumerate}

		\end{lemma}
		\begin{proof}
			First we show \eqref{p_lemma:borel}. Note that the graph $\restriction{\linegraph{}{L}{}}{B'}$ is a disjoint union of finite paths. Fix a Borel linear ordering $<$ of $X$. We will define $\phi'$ so that the value $\phi'(x)$ will only depend on
			\begin{enumerate}[label=(\alph*)]
				
				\item \label{d:previous} the values of $\phi$ on the $\restriction{\linegraph{}{L}{}}{B'}$-component of $x$,
				\item \label{d:enumeration}the index of $x$ in the unique enumeration $(v^x_i)_{i\leq m^x}$ of the $\restriction{\linegraph{}{L}{}}{B'}$-component of $x$, with the property that $v^x_0<v^x_{m^x}$ and $\forall i<m^x$ we have $(v^x_i,v^x_{i+1}) \in \linegraph{}{L}{}$.
			\end{enumerate}

			{\sc Claim.} For a connected component of $\restriction{\linegraph{}{L}{}}{B'}$ let $(v_i)_{i<m}$ be the enumeration described in \ref{d:enumeration}. There exists a homomorphism $\psi$ of  $\restriction{\linegraph{}{L}{}}{\{v_0,\dots,v_{m}\}}$ to $\linegraph{}{L}{c,n+1}$ so that we have $\restriction{\pi_{c,n,n+1}\circ\psi}{\{v_0,\dots,v_{m}\} \cap B}=\restriction{\phi}{\{v_0,\dots,v_{m}\} \cap B}$. 
			\begin{proof}[Proof of the Claim.]
				In order to see that such a homomorphism $\psi$ exists, note that the set $\{v_0,\dots,v_{m}\}$ decomposes into connected components of $\restriction{\linegraph{}{L}{}}{B}$ and paths connecting them: more precisely, there are an odd number $l$, a sequence $0 \leq i_0<i_1<\dots<i_l \leq m$ with the property that if $i \in [0,i_0) \cup (i_1,i_2) \cup (i_3,i_4) \cup \dots \cup (i_l,m]$ (where the first and last intervals could be empty) then $v_i \in B' \setminus B$, while for every $j<l$ even, $\{v_{i_j},\dots,v_{i_{j+1}}\}$ is an $\restriction{\linegraph{}{L}{}}{B}$-component.
				
				Define $\psi(v_i)$ for $i \in [i_0,\dots,i_1]$ to be $\phi(v_i) \concatt (0)$, and extend this to a homomorphism from $\restriction{\linegraph{}{L}{}}{\{v_0,\dots,v_{i_1}\}}$ to $\linegraph{}{L}{c,n+1}$. Now, assume that $\psi$ has been defined on $\{v_i:i\leq i_j\}$ for $j<l-1$ odd with $\restriction{\pi_{c,n,n+1} \circ \psi}{\{v_i:i\leq i_j\}\intersection B}=\restriction{\phi}{\{v_i:i\leq i_j\}\intersection B}$ remaining true on these vertices. We will extend $\psi$ to $\{v_i:i\leq i_{j+2}\}$. Since $\pi_{c,n,n+1} (\psi(v_{i_j}))=\phi(v_{i_j})$ holds, $\psi(v_{i_j})$ has the form $\phi(v_{i_j}) \concatt (\varepsilon)$ for some $\varepsilon \in \{0,1\}$. Note that by $\mgs(B)>2 \cdot \length(\linegraph{}{L}{c,n+1})$, we have that $i_{j+1}-i_j> 2 \cdot \length(\linegraph{}{L}{c,n+1})$. 
				
				If the parity of the distance of $\phi(v_{i_j})$ and $\phi(v_{i_{j+1}})$ in $\linegraph{}{L}{c,n}$ is the same as the parity of $i_{j+1}-i_j$ then by $i_{j+1}-i_j> 2 \cdot \length(\linegraph{}{L}{c,n+1})> \length(\linegraph{}{L}{c,n})$, the map $\psi$ extends to a homomorphism from $\restriction{\linegraph{}{L}{}}{\{v_0,\dots,v_{i_{j+1}}\}}$ to $\linegraph{}{L}{c,n+1}$ with $\psi(v_{i_{j+1}})=\phi(v_{i_{j+1}}) \concatt (\varepsilon)$. In this case define $\psi(v_i)=\phi(v_i) \concatt (\varepsilon)$ if $i \in [i_{j+1},i_{j+2}]$.
				
				Otherwise, if the parity is different, using $i_{j+1}-i_j> 2 \cdot \length(\linegraph{}{L}{c,n+1})$ again and the fact that the distance of $\phi(v_{i_{j+1}}) \concatt (0)$ and $\phi(v_{i_{j+1}}) \concatt (1)$ is odd in $\linegraph{}{L}{c,n}$ (see \eqref{c:oddistance} of Claim \ref{c:basic}) we have that $\psi$ can be extended to a homomorphism from $\restriction{\linegraph{}{L}{}}{\{v_0,\dots,v_{i_{j+1}}\}}$ to $\linegraph{}{L}{c,n+1}$ with $\psi(v_{i_{j+1}})=\phi(v_{i_{j+1}}) \concatt (1-\varepsilon)$. In this case define $\psi(v_i)=\phi(v_i) \concatt (1-\varepsilon)$ if $i \in [i_{j+1},i_{j+2}]$.
				
				This inductive process yields a homomorphism from $\restriction{\linegraph{}{L}{}}{\{v_0,\dots,v_{i_l}\}}$ to $\linegraph{}{L}{c,n+1}$ with $\restriction{\pi_{c,n,n+1}\circ\psi}{\{v_0,\dots,v_{m}\}}=\restriction{\phi}{\{v_0,\dots,v_{i_l}\}}$, and this of course can be extended to a homomorphism to the set $\{v_{i_l},\dots, v_m\}$ (which is disjoint from $B$). This finishes the proof of the claim.
			\end{proof}
			
			For a given $m$ there are only finitely many homomorphisms from the path of length $m$ to $L_{c,n+1}$. Fix an enumeration of those homomorphisms $(\psi^m_j)_{j<l_m}$ for each $m \in \N$. Now, for an $x \in B'$ let $(v^x_i)_{i \leq m^x}$ be the enumeration described in \ref{d:enumeration}, and $j^x$ minimal index for which $\psi^{m^x}_{j^x}$ satisfies the Claim. It is clear that the map $x \mapsto \psi^{m^x}_{j^x}$ is Borel, and so is the map $\phi'(x)=\psi^{m^x}_{j^x}(x)$. Moreover, $\phi'$ satisfies that it depends only on \ref{d:previous} and \ref{d:enumeration} and the requirements of the lemma. This finishes the proof of \eqref{p_lemma:borel}.
			
			Now assume that the assumptions of \eqref{p_lemma:continuous} hold, and let $<$ be the lexicographic ordering on $\linegraph{b}{L}{c}$. It is enough to check that the map $\phi'$ defined as in the first part is a continuous mapping. For a given $x$ the value $\phi'(x)$ depends only on finitely many values. Hence, it suffices to show that if $x_n \to x$ then the values determining $\phi'(x_n)$ converge to the values determining $\phi'(x)$. 
			
			From the definition of $\linegraph{b}{L}{c^0}$ it follows that a connected component of $\restriction{\linegraph{b}{L}{c^0}}{B'}$ contains the points of the form $\{(l,m,t \concatt r):l < k', m \leq c^0(l), t\in 2^{k'-l-1}\}$ for some $r \in 2^\N$. Moreover, if $(l^n_i,m^n_i,t^n_i \concatt r^n) \to  (l_i,m_i,t_i \concatt r)$ for $i<2$ with $l_i < k'$, then 
			$(l_i,m_i,t_i \concatt r)_{i<2} \in \linegraph{b}{L}{c}$ holds iff $(\pi_{c^0,k'-1}(l_i,m_i,t_i \concatt r))_{i<2} \in \linegraph{}{L}{c^0,k'-1}$ iff $(\pi_{c^0,k'-1}(l^n_i,m^n_i,t^n_i \concatt r^n))_{i<2} \in \linegraph{}{L}{c^0,k'-1}$ is true for every large enough $n$. This, and the fact that $<$ is open, implies that if $(v^{x_n}_i)_{i\leq m^{x_n}}$ and $(v^{x}_i)_{i\leq m^x}$ are the enumerations of $\restriction{\linegraph{b}{L}{c^0}}{B'}$-components described in \ref{d:enumeration}, then $m^{x_n}$ must stabilize to $m_x$ and $v^{x_n}_i \to v^{x}_i$ holds for all $i \leq m_x$. Hence, since $B$ is clopen and $\phi$ is continuous, we get that $\phi'(x_n)=\phi'(x)$ for every large enough $n$.  
		\end{proof}
		
		\begin{proof}[Proof of Proposition \ref{p:large_homo}] We define a sequence $(k_n,\phi_n)_{n \in \N}$ inductively. For convenience, we will assume that $B_0=\emptyset$. Choose $k_0=0$, then $B_{k_0}=\emptyset$, $\phi_0=\emptyset$ and $\mgs(B_{k_0})=\aleph_0>2\cdot \lgth(\linegraph{}{L}{c,1})$. Now assume that $(\phi_i,k_i)_{i \leq n}$ had already been defined with the properties that $\mgs(B_{k_n})>2 \cdot \lgth(\linegraph{}{L}{c,n+1})$ and $\phi_i$ is a Borel homomorphism from $\restriction{\linegraph{}{L}{}}{B_{k_{i}}$ to $\linegraph{}{L}{c,i}}$. Choose $k_{n+1}$ so large that  $\mgs(B_{k_{n+1}})>2 \cdot \lgth(\linegraph{}{L}{c,n+2})$. An application of \eqref{p_lemma:borel} of Lemma \ref{l:iterate} to $B_{k_n}, B_{k_{n+1}}, n$,  and $\phi_{n}$ yields a homomorphism $\phi_{n+1}$ of $\restriction{\linegraph{}{L}{}}{B_{k_{n+1}}$ to $\linegraph{}{L}{c,n+1}}$ so that $\restriction{\pi_{c,n,n+1} \circ\phi_{n+1}}{B_{k_n}}= \phi_n$. Thus, we obtain a sequence $(k_n,\phi_n)_{n \in \N}$ that satisfies the assumptions \eqref{p:union}-\eqref{p:compatible} of Claim \ref{c:bhom}, which finishes the proof of the first part.
		
		Finally, a similar proof yields the second half: first, note that that if $B=\{(l,m,r)\in \underlyingspace{b}{X}{c^0}: l <k\}$ then $\mgs(B)=\min\{c(i)+1:i \geq k\}$. This, and the assumption that $c^0(n) \to \infty$ allow us to find the sequence $(B_{k_n})_{n \in \N}$ and iterate \eqref{p_lemma:continuous} of Lemma \ref{l:inductive}. This yields a sequence $(k_n,\phi_n)_{n \in \N}$  satisfying \eqref{p:union}-\ref{p:continuity} of Claim \ref{c:bhom}.
	\end{proof}
	
	Combining the preceding theorems we obtain the following result, which of course implies Theorem \ref{t:main_intro}. 
	
	\begin{theorem}
		\label{t:main} Assume that $G$ is an analytic graph on a Hausdorff space. Then exactly one of the following holds.
		\begin{enumerate}
			\item \label{c:chromatic2} $\Borelchromatic{G} \leq 2$, i.e., $G$ is Borel bipartite.
			\item \label{c:homomorphism} $\linegraph{b}{L}{0}$ admits a continuous homomorphism to $G$.
		\end{enumerate}
	\end{theorem}
	\begin{proof}
		The fact that \eqref{c:chromatic2} and \eqref{c:homomorphism} are mutually exclusive follows from the observations $\Borelchromatic{\linegraph{b}{L}{0}}>2$ (\eqref{c:category} of Claim \ref{c:basic}) and that $\linegraph{b}{L}{0} \leq_c G$ implies $\Borelchromatic{\linegraph{b}{L}{0}} \leq \Borelchromatic{G}$. 
		
		Now, assume that \eqref{c:chromatic2} is false. Recall that $\linegraph{b}{L}{0}=\linegraph{b}{L}{c^0}$, where $c^0(0)=1$, and  $c^0(n)=2n-1$. Then by Theorem \ref{t:g0style} there exists an \oddpair{\aleph_0} $b=(c,d)$, such that $\digraph{b}{L}{b} \leq_c G$. But then $\linegraph{b}{L}{c} \leq_c G$, and using \eqref{p:continuous_homo} of Proposition \ref{p:large_homo} we obtain $\linegraph{b}{L}{0} \leq_c \linegraph{b}{L}{c} \leq_c G$.
	\end{proof}
	
	We conclude this section with proving Theorem \ref{t:large_gaps}, that is:
	
	\begin{largegaps}
		
		Suppose that $X$ is a standard \Borel space and $L$ is an acyclic \Borel graph on $X$ of vertex degree at most two. Then the following are equivalent:
		\begin{enumerate}
			\item There is a \Borel homomorphism from $L$ to every \Borel graph $G$ of \Borel chromatic number at least three.
			\item \label{c:largegapmodulo2} The graph $L$ has the large gap property modulo a two-colorable set.
		\end{enumerate}
		
	\end{largegaps}
	\begin{proof}
		Assume first \eqref{c:largegapmodulo2}. Using Claim \ref{cl:large} together with \eqref{p:borel_homo} of \ref{p:large_homo} we obtain a sequence $(B_n)_{n \in \N}$ of Borel sets and an $L$-invariant Borel set $M$ such that $\restriction{L}{\bigcup_n B_n} \leq_B \linegraph{b}{L}{0}$, $\restriction{\connecteder{L}}{X \setminus(M \cup \bigcup_{n} B_n)}$ is smooth, and $X \setminus \bigcup_{n} B_n$ is $L$-invariant, and $\Borelchromatic{\restriction{\linegraph{}{L}{}}{M}}\leq 2$. By Claim \ref{cl:basiclines} and the invariance of $M$ we have that  $\Borelchromatic{\restriction{\linegraph{}{L}{}}{X \setminus \bigcup_{n} B_n}} \leq 2$, so $\restriction{\linegraph{}{L}{}}{X \setminus \bigcup_{n} B_n}$ admits a Borel homomorphism to each non-empty Borel graph. Putting together the Borel homomorphisms on the invariant sets $X \setminus \bigcup_{n} B_n$ and $\bigcup_n B_n$ we obtain $L \leq_B \linegraph{b}{L}{0}$. Thus, by Theorem \ref{t:main}, $L$ admits a Borel homomorphism to each Borel graph with Borel chromatic number $>2$. 
		
		Now assume that $L \leq_B \linegraph{b}{L}{0}$, witnessed by the Borel map $\phi$. Let $M$ be the set from \eqref{c:ontohomo} of Claim \ref{cl:basiclines} and let $B \subseteq \underlyingspace{b}{X}{0}$ witness that $\linegraph{b}{L}{0}$ has the large gap property (\eqref{c:largegap} of Claim \ref{c:basic}).  To show the theorem, it is enough to check that $\restriction{\linegraph{}{L}{}}{X \setminus M}$ has the large gap property. Let $B'= (\restriction{\phi}{X \setminus M})^{-1}(B)$, it is easy to see from the fact that $\restriction{\phi }{X \setminus M}$ maps $\linegraph{}{L}{}$-components onto $\linegraph{b}{L}{0}$-components that  $B'$ witnesses the large gap property of $\restriction{\linegraph{}{L}{}}{X \setminus M}$.
	\end{proof}

	\section{An antibasis result for digraphs}
	\label{s:digraphs}
	Finally, we show a slightly more general version of Theorem \ref{t:antibasis}, that is:
	
	\begin{theorem}
		\label{t:antibasis2}
		Suppose that $\digraph{}{G}{}$ is an analytic digraph on a Hausdorff space
		with $\Borelchromatic{\digraph{}{G}{}}>2$. Then there is a sequence
		$\sequence{\digraph{}{L}{t}}[t \in \Cantorspace]$ of \Borel oriented graphs on standard Borel spaces
		such that for each $t \in \Cantorspace$ we have $\digraph{}{L}{t} \leq_c \digraph{}{G}{}$, and $\Borelchromatic{\digraph{}{L}{t}}>2$, and any
		Borel graph that admits a \Borel
		homomorphism to at least two oriented graphs of the form $\digraph{}{L}{t}$
		has a \Borel two-coloring.
	\end{theorem}

	Let us start with some definitions. Assume that $\digraph{}{G}{}$ is an acyclic oriented graph on a space $X$, and let $B \subseteq X$. Using Claim \ref{c:distances} we can define the \definedterm{didistance set} of $B$ by letting $\distanceset{B}{\digraph{}{G}{}}=\{n \in \Z: \exists x,y \in B \ \didist{x}{y}{\digraph{}{G}{}}=n\}$. 
	
	\begin{lemma}
		\label{l:fromcanonical}
		Assume that $\digraph{}{L}{}$ is a $\leq 2$-regular acyclic Borel oriented graph on the space $X$, $b=(c,d)$ is an \oddpair{\aleph_0}, and $\digraph{b}{L}{b} \leq_B \digraph{0}{L}{}$. Assume that $C$ is a Borel $\connecteder{\digraph{}{L}{}}$-complete set. There exists a non-meager Borel set $B \subseteq \underlyingspace{b}{X}{c}$ such that $\distanceset{B}{\digraph{b}{L}{b}} \subseteq  \distanceset{C}{\digraph{}{L}{}}$.
	\end{lemma}
	\begin{proof}
		Let $\phi$ be a Borel homomorphism from $\digraph{b}{L}{b}$ to $\digraph{0}{L}{}$, and let $M=\{x \in \underlyingspace{b}{X}{c}: \phi \text{ mapping }[x]_{\connecteder{}} \to [\phi(x)]_{\connecteder{\linegraph{}{L}{}}} \text{ is not onto}\}$. By Claim \ref{cl:basiclines} we have $\Borelchromatic{\restriction{\digraph{b}{L}{b}}{ M}} \leq 2$. By the invariance of $M$ and \eqref{c:category} of Claim \ref{c:basic} it must be meager. Define $B=\phi^{-1}(C) \setminus M$, we check that $B$ is non-meager. Note that, as $C$ is an $\connecteder{\digraph{}{L}{}}$-complete set,  the set $B$ is a $\restriction{\connecteder{c}}{(\underlyingspace{b}{X}{c} \setminus M)}$-complete. As $[B]_{\connecteder{c}} \supseteq \underlyingspace{b}{X}{c} \setminus M$ is co-meager, it follows from \eqref{c:meagersaturation} of Claim \ref{c:basic} that $B$ cannot be meager.

		Finally, if $x,y \in B$, let $p=((z_0,\dots,z_l),d_p)$ be an $\digraph{b}{L}{b}$-path with $z_0=x$ and $z_l=y$. Then, $p'=((\phi(z_0),\dots,\phi(z_l)),d_p)$ is an $\digraph{}{L}{}$-path, with $\dlength(p')=\dlength(p)$. It follows that $\didist{\phi(x)}{\phi(y)}{\digraph{b}{L}{b}} \in\distanceset{C}{\digraph{}{L}{}}$. Thus, $\distanceset{B}{\digraph{b}{L}{b}} \subseteq  \distanceset{C}{\digraph{}{L}{}}$.
	\end{proof}
	
	In order to carry out our construction we will impose a growth condition on the approximations to our graphs. Assume that $b=(c,d)$ is an \oddpair{\aleph_0}. We say that \definedterm{$b$ has property $(*)$} if for every $i \in \N $ we have \[\Sigma(d(i))> 8 \cdot \sum_{j<i} 2^{i-j} \cdot \absolutevalue{\Sigma(d(j))}.\] 
	\begin{lemma}
		\label{l:crucial}
		Suppose that $b=(c,d)$ is an \oddpair{\aleph_0} with property $(*)$. Then there exists a collection $(P_t)_{t \in 2^\N}$ of perfect subsets of $\underlyingspace{b}{X}{c}$ such that for every $t\neq t'$ we have
		\begin{enumerate}
			\item \label{c:infinitesets} $\absolutevalue{\distanceset{P_t}{\digraph{b}{L}{b}}}=\aleph_0$,

			\item \label{c:finiteintersection} There exists an $i_0 \in \N$ such that for every $k,k' \geq i_0$ if  $k \in \distanceset{P_t}{\digraph{b}{L}{b}}$ and $k' \in \distanceset{P_{t'}}{\digraph{b}{L}{b}}$ then $\frac{k}{k'} \not \in [\frac{1}{4},4]$. 
			
			\item \label{c:chromaticnumbers} $\Borelchromatic{\restriction{\digraph{b}{L}{b}}{[P_t]_{\connecteder{c}}}}=3$.
		\end{enumerate}
		
	\end{lemma}
	\begin{proof}
		Let $S \subseteq 2^\N$ be a perfect almost disjoint family of infinite sets (identifying $2^\N$ with $\mathcal{P}(\N)$). Of course (using a bijection between $S$ and $2^\N$), it is enough to construct a family indexed by the elements of $S$. 	
		
		For $t \in  S$, let \[P_t=\{(n,k,r) \in \underlyingspace{b}{X}{c}:n=k=0, \forall i \in \N \ (t(i)=0 \implies r(i)=0)\}.\] 
		
		\noindent
		{\sc Claim.} Assume that $x \neq y \in P_t$ and $x\connecteder{c}y$. Let $x=(0,0,r^x), y=(0,0,r^y)$, and let $i \in \N$ be maximal with with $r^x(i) \neq r^y(i)$ (such an $i$ exists by $x\connecteder{c}y$). Then $t(i)=1$ and $\absolutevalue{\didist{x}{y}{\digraph{b}{L}{b}}} \in [\frac{\Sigma(d(i))}{2},2\cdot \Sigma(d(i))]$.
		
		\begin{proof}
			It is obvious from the definition of $P_t$, that we have $t(i)=1$.
			
			Moreover, it follows from the definition of $\digraph{b}{L}{b}$ and the choice of $i$ that $\didist{x}{y}{\digraph{b}{L}{b}}=\didist{\pi_{c,i}(x)}{\pi_{c,i}(y)}{\digraph{}{L}{b,i}}$, so it is enough to give an estimation on the latter. Since $r^x(i) \neq r^y(i)$, $\pi_{c,i}(x)$ and $\pi_{c,i}(y)$ are in different copies of $\digraph{}{L}{b,i-1}$ in $\digraph{}{L}{b,i}$. But then,
			
			\begin{multline*}\Sigma(d(i))-\absolutevalue{\dlength(\digraph{}{L}{b,i-1})})\leq \absolutevalue{\didist{\pi_{c,i}(x)}{\pi_{c,i}(y)}{}}\\ \leq \Sigma(d(i))+\absolutevalue{\dlength(\digraph{}{L}{b,i-1})}.\end{multline*}
			So, by an easy induction we have
			\begin{multline*}\Sigma(d(i))-\sum_{j <i} 2^{i-j} \cdot \Sigma(d(j)) \leq \absolutevalue{\didist{\pi_{c,i}(x)}{\pi_{c,i}(y)}{}}\\ \leq \Sigma(d(i))+\sum_{j <i} 2^{i-j} \cdot \Sigma(d(j)),\end{multline*}
			which implies our statement by $(*)$.
		\end{proof}
		
		The Claim clearly implies property \eqref{c:infinitesets}.
		
		Assume that $t \not = t'$ are given. By the choice of $S$ there exists an $i^*_0 \in \N$ such that $t \cap t' \subseteq i^*_0$, let $i_0=\frac{\Sigma(d(i^*_0))}{2}$ and assume that $k,k' \geq i_0$ with $k \in \distanceset{P_t}{\digraph{b}{L}{b}}, k' \in \distanceset{P_{t'}}{\digraph{b}{L}{b}}$.  The choice of $i^*_0$ and $i_0$ together with the Claim and $(*)$ yields that $k \in [\frac{\Sigma(d(i))}{2},2\cdot \Sigma(d(i))]$ and $k' \in [\frac{\Sigma(d(i'))}{2},2\cdot \Sigma(d(i'))]$ with $i \neq i'$. But then $\frac{k}{k'} \not \in [\frac{1}{4},4]$ follows from $(*)$, showing \eqref{c:finiteintersection}.
		
		Finally, a Baire category argument analogous to the one in the proof of \eqref{c:category} of Claim \ref{c:basic} yields that \eqref{c:chromaticnumbers} holds for each $t \in S$. 
	\end{proof}
	
	\begin{lemma}
		\label{l:saturation}
		Assume that $b^*=(c^*,d^*)$ is an \oddpair{\aleph_0} with property $(*)$, and $B \subseteq \underlyingspace{b}{X}{c^*}$ is a Borel set so that the set $[B]_{\connecteder{c^*}}$ is co-meager. Then there exists an $i_1 \in \N$ such that for every $i > i_1$ we have that $ \distanceset{B}{\digraph{b}{L}{b^*}} \cap [\frac{\Sigma(d^*(i))}{2},2 \cdot \Sigma(d^*(i))] \neq \emptyset$.
	\end{lemma}
	\begin{proof}
		By our assumption on $B$ and \eqref{c:meagersaturation} of Claim \ref{c:basic}, we can find an $i_1$ and a non-empty basic open set of the form $[(i_1,k,\sigma)]$ in which $B$ is co-meager. Let $i>i_1$. By shrinking $B$ with the $\connecteder{c^*}$ -saturation of the meager set $[(i,k,\sigma)] \setminus B$ (which is also a meager set), we can assume that $[(i-1,k,\sigma)] \cap B=[(i-1,k,\sigma)] \cap [B]_{\connecteder{c^*}}$. In particular, we can find an $r \in 2^\N$ so that $(i-1,k,\sigma \concatt (0) \concatt r), (i-1,k,\sigma \concatt (1) \concatt r) \in B$ hold, let $x=(i-1,k,\sigma \concatt (0) \concatt r),$ and $y=(i-1,k,\sigma \concatt (1) \concatt r)$. 
		
		Again, it is clear that $\didist{x}{y}{\digraph{b}{L}{b^*}}=\didist{\pi_{c,i}(x)}{\pi_{c,i}(y)}{\digraph{}{L}{b^*,i}}$, moreover, $\pi_{c,i}(x)$ and $\pi_{c,i}(y)$ are in different copies of $\digraph{}{L}{b^*,i-1}$ in $\digraph{}{L}{b^*,i}$. So, \begin{multline*}\Sigma(d^*(i))-\sum_{j <i} 2^{i-j} \cdot \Sigma(d^*(j)) \leq \absolutevalue{\didist{\pi_{c,i}(x)}{\pi_{c,i}(y)}{}} \\ \leq \Sigma(d^*(i))+\sum_{j <i} 2^{i-j} \cdot \Sigma(d^*(j)),\end{multline*}
		which implies our statement by $(*)$.
	\end{proof}
	
	\begin{proof}[Proof of Theorem \ref{t:antibasis2}]
		
		By Theorem \ref{t:g0style} without loss of generality we can assume that $G=\digraph{b}{L}{b}$ for some \oddpair{\aleph_0} $b=(c,d)$ with property $(*)$. Now, using Lemma \ref{l:crucial} we obtain a family $(P_t)_{t \in  2^\N}$ of perfect subsets of $\underlyingspace{b}{X}{c}$ having properties \eqref{c:infinitesets}--\eqref{c:chromaticnumbers}. For each $t \in  2^\N$ let $\digraph{}{L}{t}=\restriction{\digraph{b}{L}{b}}{[P_t]_{\connecteder{c}}}$. We show that $(\digraph{}{L}{t})_{t \in 2^\N}$ satisfies the requirements of the theorem. The condition on the Borel chromatic numbers is clear from \eqref{c:chromaticnumbers} of Lemma \ref{l:crucial}.

		Let $t,t' \in 2^\N$ be  distinct. Assume that $\digraph{a}{H}{}\leq_B \digraph{a}{L}{t}, \digraph{a}{L}{t'}$ with $\Borelchromatic{\digraph{}{H}{}}=3$. Then, by Theorem \ref{t:g0style} we can assume that $\digraph{a}{H}{}=\digraph{b}{L}{b^*}$ and that $b^*$ has property $(*)$.  As $P_t$ and $P_{t'}$ are $\connecteder{\digraph{}{L}{t}}$ and $\connecteder{\digraph{}{L}{t'}}$-complete sets, using Lemma \ref{l:fromcanonical} we obtain non-meager Borel sets $B,B'$ in $\underlyingspace{b}{X}{c^*}$, with $\distanceset{B}{\digraph{b}{L}{b^*}} \subseteq \distanceset{P_t}{\digraph{}{L}{t}}$ and $\distanceset{B'}{\digraph{b}{L}{b^*}} \subseteq \distanceset{P_{t'}}{\digraph{}{L}{t'}}.$ Let $i>i_0,i_1,i'_1$, where $i_0$ comes from \eqref{c:finiteintersection} of Lemma \ref{l:crucial}, while $i_1,i'_1$ are obtained from applying Lemma \ref{l:saturation} to $B$ and $B'$.
		
		By Lemma \ref{l:saturation} we can find $k \in \distanceset{B}{\digraph{b}{L}{b^*}} \cap [\frac{\Sigma(d^*(i))}{2},2 \cdot \Sigma(d^*(i)))]$, $k' \in \distanceset{B'}{\digraph{b}{L}{b^*}} \cap [\frac{\Sigma(d^*(i))}{2},2 \cdot \Sigma(d^*(i))]$. But then $ \frac{k}{k'} \in [\frac{1}{4},4]$, which contradicts $k \in \distanceset{P_t}{\digraph{}{L}{t}}, k' \in \distanceset{P_{t'}}{\digraph{}{L}{t'}}$ and  \eqref{c:finiteintersection} of Lemma \ref{l:crucial}.
	\end{proof}
	
	\section{Open problems}
	\label{s:problems}	
	
	We conclude with a number of open problems. First, it is not clear, how Theorem \ref{t:large_gaps} can be generalized to arbitrary Borel graphs. 
	
	\begin{problem}
		Characterize the Borel graphs with Borel chromatic number $3$, which admit a Borel homomorphism to each Borel graph $G$ with $\Borelchromatic{G}>2$, (or, equivalently, the ones which are $\leq_B \linegraph{b}{L}{0}$). 
	\end{problem}  
	
	The \definedterm{product} of graphs $G$ on $X$ and $G'$ on $X'$ is the graph on $X \times X'$ given by $((x,x') , (y,y')) \in G \times G' \iff (x,y) \in G \text{ and } (x',y') \in G'$. The Borel version of Hedetniemi's conjecture reads as follows: Is it the case that $\Borelchromatic{G \times G'} = \min \{\Borelchromatic{G}, \Borelchromatic{G'}\}$?
	
	Theorem \ref{t:main_intro} implies that the answer is affirmative, if $\min\{\Borelchromatic{G},\Borelchromatic{G'}\} \leq 3$. El-Zahar and Sauer \cite{zahar} showed that for finite graphs the bound $4$ already implies an affirmative answer. Hence the following problem is quite natural.
	
	\begin{problem}
		Assume that $G,G'$ are Borel graphs on standard Borel spaces, and 
		$\min\{\Borelchromatic{G},\Borelchromatic{G'}\} \leq 4$. Is it true that $\Borelchromatic{G \times G'} = \min \{\Borelchromatic{G}, \Borelchromatic{G'}\}$?
		
	\end{problem}
	
	Note that a recent breakthrough result of Shitov \cite{shitov2019counterexamples} is that the answer is negative in general, there exists a counterexample for finite graphs.
	
	The $\mathbb{G}_0$-dichotomy, the results in \cite{toden}, and the current paper give a complete description of the existence of simple bases for Borel graphs with a given Borel chromatic number. However, the natural reformulation of the notion of chromatic numbers in terms of homomorphism raises the following problem:
	\begin{problem}
		Characterize the Borel graphs $H$ so that the collection $\{G:G \text{ is a Borel graph, } G \not \leq_B H\}$ has a single element basis.
	\end{problem}
	
	It is conceivable that such a characterization is impossible due to a complexity barrier.  
	
	Babai's celebrated results \cite{babai2016graph} suggest that among finite graphs the isomorphism relation is simpler than the homomorphism relation. It would be interesting to know the answer to the analogous question in the case of Borel graphs. 
	
	\begin{problem}
		Determine the projective complexity of the isomorphism relation on Borel graphs on Polish spaces.
	\end{problem}

	\bibliographystyle{abbrv}
	\bibliography{bibliography}
	
\end{document}